\documentclass{amsart}
\usepackage[utf8]{inputenc}
\usepackage{xcolor}

\usepackage{enumerate} 
\usepackage{upref}
\usepackage{verbatim} 
\usepackage{color}
\usepackage{graphicx}
\usepackage{mathtools,amsmath,amsfonts,amssymb,amsthm, bbm, dsfont}
\usepackage{hyperref}

\usepackage[textsize=tiny]{todonotes}

\newcommand*{\mailto}[1]{\href{mailto:#1}{\nolinkurl{#1}}}

\usepackage{amssymb, amsmath, amsfonts, amsthm}
\usepackage[all]{xy}

\usepackage{marginnote}

\newcommand{\eps}{\epsilon}
\newcommand{\R}{\mathbb{R}}

\newcommand{\N}{\mathbb{N}}

\newcommand{\mA}{\mathcal{A}}
\newcommand{\mB}{\mathcal{B}}
\newcommand{\mD}{\mathcal{D}}
\newcommand{\mF}{\mathcal{F}}
\newcommand{\mG}{\mathcal{G}}
\newcommand{\Cc}{C_c^\infty}

\newcommand{\supp}{\text{\normalfont{supp}}}

\newcommand{\hlf}{\frac{1}{2}}
\newcommand{\bbo}{\mathds{1}}

\newcommand{\id}{\text{\normalfont{id}}}
\newcommand{\Lag}{\hat{L}} 
\newcommand{\Eul}{\hat{M}} 

\theoremstyle{plain}
\newtheorem{thm}{Theorem}[section]
\newtheorem{lem}[thm]{Lemma}
\newtheorem{prop}[thm]{Proposition}
\newtheorem{defn}[thm]{Definition}

\newtheorem{exmp}[thm]{Example}
\newtheorem{cor}[thm]{Corollary}
\theoremstyle{plain}

\newtheorem*{note}{Note}

\usepackage{graphicx}
\usepackage{caption}
\usepackage{subcaption}
\graphicspath{ {./figs/} }
\usepackage[outdir=./figs/]{epstopdf}

\usepackage{mdframed}
\usepackage{marginnote}
\newmdenv[
  topline=false,
  bottomline=false,
  skipabove=\topsep,
  skipbelow=\topsep,
  leftmargin=-10pt,
  rightmargin=-10pt,
  innertopmargin=0pt,
  innerbottommargin=0pt
]{siderules}

\begin{document}
\title[A Lipschitz Metric for the HS equation]{A Lipschitz Metric for $\alpha$-Dissipative Solutions to the Hunter--Saxton Equation}

\author[K. Grunert]{Katrin Grunert}
\address{Department of Mathematical Sciences\\ NTNU Norwegian University of Science and Technology\\ NO-7491 Trondheim\\ Norway}
\email{\mailto{katrin.grunert@ntnu.no}}
\urladdr{\url{https://www.ntnu.edu/employees/katrin.grunert}}

\author[M. Tandy]{Matthew Tandy}
\address{Department of Mathematical Sciences\\
NTNU Norwegian University of Science and Technology\\
NO-7491 Trondheim\\ Norway}
\email{\mailto{matthew.tandy@ntnu.no}}
\urladdr{\url{https://www.ntnu.edu/employees/matthew.tandy}}

\thanks{We acknowledge support by the grants {\it Waves and Nonlinear Phenomena (WaNP)} and {\it Wave Phenomena and Stability - a Shocking Combination (WaPheS)}  from the Research Council of Norway. }  

\begin{abstract}
We explore the Lipschitz stability of solutions to the Hunter--Saxton equation with respect to the initial data. In particular, we study the stability of $ \alpha $-dissipative solutions constructed using a generalised method of characteristics approach, where $\alpha$ is a function determining the energy loss at each position in space.
\end{abstract}

\maketitle

\section{Introduction}
In this work we study particular solutions to the Hunter--Saxton equation, which is given by
\begin{equation}\label{eqn:HS}\tag{HS}
	u_t(x,t) + uu_x(x,t) 
	= \frac{1}{4}\Bigg(\int_{-\infty}^x u_x^2(y,t)\, dy - \int_x^{+\infty}u_x^2(y,t)\, dy\Bigg).
\end{equation}
To be precise, our goal is to define a metric for which $\alpha$-dissipative solutions, constructed using a generalised method of characteristics, are Lipschitz continuous with respect to the initial data.

Equation \eqref{eqn:HS} was first introduced by Hunter and Saxton to model nonlinear instability in the director field for nematic liquid crystals \cite{MR1135995}. The physical applications of this equation are not the focus of this paper, however.

Solutions to this equation may experience singularities in finite time. 
Specifically, a solution $u$ will remain bounded and continuous, 
while its spatial derivative will diverge to $-\infty$ at certain points.
Parts of the energy, calculated using the energy density function $u_x^2$,
initially spread over sets of positive measure,  
will concentrate onto sets of Lebesgue measure zero.
These singularities are referred to as ``wave-breaking'', and how one
treats the concentrated energy after these points in time determines the solution. 

Discarding the concentrated energy, one obtains dissipative solutions, 
for which existence and uniqueness have been shown \cite{MR2191785,MR2796054}. 
On the other hand, one could retain the energy, 
obtaining so called conservative solutions, 
in which case existence has been shown in \cite{MR2653980, MR3573580}, 
and uniqueness in \cite{MR4400175}. 
Finally, one could choose to remove an $\alpha$ proportion of the energy, 
with $\alpha \in [0, 1]$. 
These are known as $\alpha$-dissipative solutions, 
for whom existence has been established in \cite{MR3860266}. 

This paper focuses on the importance of the energy in the notion of a solution
to our problem. To be precise, an $\alpha$-dissipative solution to the Cauchy problem of \eqref{eqn:HS} is a pair $(u, \mu)$ satisfying
\begin{subequations}
\begin{align}
	u_t + u u_x 
	&= \frac{1}{4} 
	\left(
		\int_{-\infty}^x d\mu - \int_x^{+\infty} d\mu	
	\right)\\
	\mu_t + (u \mu)_x 
	&\leq 0,\label{ineq:measIneq}
\end{align}
\end{subequations}
in the distributional sense. 
The second measure valued PDE inequality tracks the energy, 
and correspondingly the variable $\mu$ is a positive
finite Radon measure representing the current energy.

To motivate where equation \eqref{ineq:measIneq} comes from,
formally consider $u \in C^2(\R\times [0, +\infty)) $, such that $\mu=u_x^2(\cdot, t) \in L^2(\R)$ for all times $t\geq 0$. Then
\begin{equation}\label{eqn:SmMPDE}
	(u_x^2)_t 
	= 2u_x u_{xt}
	= 2u_x ( - (uu_x)_x + \hlf u_x^2)
	= -u_x^3 - 2u u_x u_{xx} = -(uu_x^2)_x.
\end{equation}
In other words, equation \eqref{ineq:measIneq} is satisfied with equality, and
$\mu(t, \R) = \mu(0, \R)$ for all times $t \geq 0$. 
This is thus a fully conservative solution. 
In reality, global solutions experience weaker regularity than we have assumed here, due to wave-breaking. Furthermore, discarding part of the concentrated energy at wave breaking yields a loss of energy resulting in \eqref{ineq:measIneq}.

The prequel to this piece of work \cite{grunert2021lipschitz} takes $\alpha$ to be constant, and a Lipschitz metric in time was constructed. However, we had to assume that the two solutions one is measuring the distance between share the same $\alpha$. This paper continues this work, constructing a new Lipschitz stable metric for the case where $\alpha$ is now possibly different for both solutions, and is a function of space. In this scenario, the amount of energy lost is determined by the point where the energy concentrates.
In particular, we are looking for a metric $d$ that satisfies the estimate
\begin{equation}\label{ineq:desired}
	d\big((u_A(t), \mu_A(t)), (u_B(t), \mu_B(t))\big) 
	\leq e^{Ct}d\big((u_A(0), \mu_A(0)), (u_B(0), \mu_B(0))\big),
\end{equation}
for all $t \geq 0$. Here $ C\in\R $ is some positive constant.
The method we use is developed from \cite{MR3573580}, where a Lipschitz metric for conservative solutions has been found using ideas from \cite{MR2653980}.
An alternative construction making use of pseudo-inverses was employed in \cite{MR3941227}. 
In \cite{MR2191785}, a metric satisfying property \eqref{ineq:desired} has also been found for dissipative solutions, in addition to Lipschitz continuity with respect to the variable $t$. This metric uses an optimal transportation approach, constructing a Wasserstein / Kantorovich-Rubenstein inspired cost function over transportation plans, and minimising over said plans. 

A generalised method of characteristics is used to construct $\alpha$-dissipative solutions to \eqref{eqn:HS} and to define a metric. Up until wave breaking occurs, we can define our Lagrangian coordinates $ (y, U, V) $ by 
\begin{subequations}
\begin{align}
 y_t(\xi, t) &= u(y(\xi, t), t), \\
U(\xi, t) & = u(y(\xi, t), t),\\ \label{def:SmoothLagVar3}
V(\xi, t) &= \int_{-\infty}^{y(\xi, t)} u^2_x(z, t)\, dz, 
\end{align}
\end{subequations}
with $ \xi $ a parameter, the so called ``label'' of the characteristic $y(\xi,t)$.
From the classical sense of Lagrangian coordinates, we may sometimes refer to $\xi$ as a ``particle''.

This leads to an ODE system, given by
\begin{subequations}\label{eqn:SmoothLagSys}
	\begin{align}
	 y_t(\xi, t) &= U(\xi, t), \label{eqn:SmoothLagSysY}\\
	 	U_t(\xi, t) &= \frac{1}{2} V(\xi, t) - \frac{1}{4} \lim_{\xi\to\infty}V(\xi, t), \label{eqn:SmoothLagSysU}\\
		V_t(\xi, t) &= 0 \label{eqn:SmoothLagSysV}.
	\end{align}
\end{subequations}
Assuming that no wave breaking occurs at time zero, one can take as initial data $y(\xi, 0) = \xi$.
The first two variables $y$ and $U$ represent respectively the position and velocity 
of particles $\xi \in \R$ as usual, while the third variable $V$ corresponds to
the $\mu$ in Eulerian variables, and represents the cumulative energy up to particle $\xi$.

To demonstrate where the third ODE comes from, once again formally consider a sufficiently smooth solution $u$ such that \eqref{eqn:SmMPDE} is satisfied. Then, differentiating \eqref{def:SmoothLagVar3} with respect to time,
\begin{align*}
	V_t(\xi, t) 
	&= U(\xi, t)u_x^2(y(\xi, t), t) + \int_{-\infty}^{y(\xi,t)} (u_x^2(x, t))_t\, dx\\
	&= U(\xi, t)u_x^2(y(\xi, t), t) - U(\xi, t)u_x^2(y(\xi, t), t)
	= 0.
\end{align*}

Wave breaking in Lagrangian coordinates corresponds to a collection of 
characteristics colliding. Specifically, wave breaking occurs at 
time $ \tau(\xi) $ for $\xi \in \R$ when $y_\xi(\xi, \tau(\xi)) = 0$.
In the case of piecewise affine and continuous solutions in Lagrangian coordinates,
this corresponds to intervals where the function $y(\cdot,\tau)$ is constant.
The desire to characterise this behaviour at time zero is what prevents us from simply taking 
$y(\xi, 0) = \xi$, as such initial data does not contain concentrated particles initially. This problem is overcome by applying the mappings between Eulerian and Lagrangian coordinates, introduced in~\cite{MR2372478} in the context of the Camassa--Holm equation.

For a given $\xi \in \R$, the wave breaking time $\tau(\xi)$ can be calculated using the initial data for the ODE system \eqref{eqn:SmoothLagSys}. In particular,
\[
	\tau(\xi) = 
	\begin{cases}
		-2\frac{y_\xi(\xi, 0)}{U_\xi(\xi, 0)}, &\mbox\ U_\xi(\xi, 0) < 0,\\
		0, &\mbox\ U_\xi(\xi, 0) = 0 = y_\xi(\xi, 0),\\
		+\infty, &\mbox\ \text{otherwise}.
	\end{cases}
\]
For a fully conservative solution the system \eqref{eqn:SmoothLagSys} determines the solution for all time. On the opposite end of the spectrum, a fully dissipative solution corresponds to the system formed by equations \eqref{eqn:SmoothLagSysY} and \eqref{eqn:SmoothLagSysU}, but setting $V_\xi(\xi ,t) = 0$ for $t \geq \tau(\xi)$.

In the general case, $\alpha: \R \to [0, 1]$, 
the energy loss at wave breaking is dependent on the particles position at time $\tau(\xi)$, 
and is given by $\alpha(y(\xi, \tau(\xi)))$. The $\alpha$-dissipative solution is thus given by \eqref{eqn:SmoothLagSysY} and \eqref{eqn:SmoothLagSysU}, and setting
\[
	V(\xi, t) = \int_{-\infty}^\xi V_\xi(\eta, 0) \big(1 - \alpha(y(\eta, \tau(\eta))) \bbo_{ \{r \in \R \mid t \geq \tau(r) > 0 \} }(\eta) \big)\, d\eta.
\]
Using this, one obtains the conservative solution in the case $\alpha \equiv 0$,
and the fully dissipative solution in the case $\alpha \equiv 1$.

The construction of our metric will take advantage of the approachable properties of these
Lagrangian coordinates. The general idea is as follows.
First, we establish how one transforms between Eulerian and Lagrangian coordinates.
Second, we define a suitable metric in Lagrangian coordinates,
satisfying Lipschitz stability with respect to the initial Lagrangian data, 
similar to inequality \eqref{ineq:desired}.
Finally, we define a suitable metric in Eulerian coordinates by measuring the distance of the corresponding Lagrangian coordinates, inheriting the Lipschitz stability we require.

The paper is organised as follows. Section 2 begins with the setup of the relevant spaces for our problem, in both Lagrangian and Eulerian coordinates.

To solve our problem we will need to introduce a secondary dummy measure $\nu$. This will also be a positive finite Radon measure, which is bounded from below by $\mu$ and which plays a key role when defining the transformations between Eulerian and Lagrangian coordinates. In Lagrangian variables this will correspond to a function $H$. Importantly, $ \nu $ is a necessity in the construction of our Lipschitz metric, but does not form part of the solution to \eqref{eqn:HS}. Therefore we will need to consider equivalence classes with respect to $\nu$ when constructing our metric in Eulerian coordinates.

Energy concentrating on sets of measure zero must be accounted for in the definition of the initial characteristics. Thus the next step in Section 2 is to 
introduce a mapping from Eulerian to Lagrangian coordinates, and vice versa, that account for this initial energy concentration. For three Eulerian coordinates, there will be a corresponding four Lagrangian coordinates. Hence there will be some redundancy, in that multiple Lagrangian coordinates will correspond to the same set of Eulerian coordinates. These will form a set of equivalence classes, related by a group of homeomorphisms called relabelling functions.

Throughout the second section we will introduce relevant established results that we make use of.

In Section 3, we construct a semi-metric in Lagrangian coordinates that satisfies Lipschitz continuity with respect to the initial data. This will form a central part of the construction of our metric.

We will see that the semi-metric we construct in Section 3 is far from optimal, since the distance between two elements of the same equivalence class, i.e. two elements representing the same Eulerian coordinates, can be positive. In Section 4, we overcome this issue and detail how we construct the metric in Lagrangian coordinates. Additionally, we establish the Lipschitz continuity with respect to the initial data in the Lagrangian setting.

In the final section, Section 5, we return to Eulerian coordinates, using our metric in Lagrangian coordinates to define a Lipschitz metric in time. In this section we have to take equivalence classes with respect to the dummy variable $\nu$ into account. 

\section{Lagrangian and Eulerian coordinates}

Before we can begin our construction of the metric, we must set up our Eulerian and Lagrangian coordinate spaces. In addition, we must examine the Lagrangian ODE problem, what it means to be a solution in Eulerian coordinates, and introduce relevant results from past literature. This follows the construction outlined in~\cite{MR2653980} and \cite{MR3860266}.

We begin by introducing an important set. Let $E$ be the Banach space of $L^\infty(\R)$ functions with $L^2(\R)$ weak derivatives, with an associated norm $\|\cdot\|_E$,
\[
	E \coloneqq \{f \in L^\infty(\R) \mid f' \in L^2(\R) \}, \quad\ \|f\|_{E} = \|f\|_\infty + \|f'\|_2.
\]

Furthermore, define $H_i \coloneqq H^1(\R) \times \R^i$ for $i = 1, 2$, and $ H_0 = L^2(\R) \times \R $, with the norms
\[
	\|(f, x)\|_{H_i} = \sqrt{\|f\|^2_{H^1} + |x|^2}, \quad\ \|(f, x)\|_{H_0} = \sqrt{\|f\|_2^2 + |x|^2}.
\]

We split the real line into two overlapping sets $(-\infty, 1)$ and $(-1, \infty)$, and pick two functions $\chi^+$ and $\chi^-$ in $C^\infty(\R)$ satisfying the following three properties,
\begin{itemize}
	\item $ \chi^- + \chi^+ = 1 $,
	\item $ 0 \leq \chi^+ \leq 1 $,
	\item $ \supp(\chi^-) \subset (-\infty ,1), $ and $ \supp(\chi^+) \subset (-1, \infty) $, 
\end{itemize}
called a partition of unity.

Using these two functions, we define the following two linear, continuous, and injective mappings,
\begin{align}\label{map:R}
	R_1: H_1 \to E&  \quad\ (f, a) \stackrel{\mathclap{R}}{\mapsto} \hat f=f + a \cdot \chi^+,\\
	R_2: H_2 \to E&  \quad\ (f, a, b) \mapsto \hat f=f + a \cdot \chi^+ + b \cdot \chi^-.
\end{align}
They define the following two Banach spaces, which are subsets of $E$,
\[
	E_1 \coloneqq R_1(H_1), \quad\ \|\hat f\|_{E_1} = \|R_1^{-1}(\hat f)\|_{H_1},
\]
\[
	E_2 \coloneqq R_2(H_2), \quad\ \|\hat f\|_{E_2} = \|R_2^{-1}(\hat f)\|_{H_2}.
\]
Note, from \eqref{map:R}, operation $R$ is well defined for elements of $H_0$. We define the set $E_0$, and the corresponding norm $\|\cdot\|_{E_0}$, by
\[
	E_0 \coloneqq R(H_0), \quad\ \|f\|_{E_0} = \|R^{-1}(f)\|_{H_0}.
\] 
Finally, our $\alpha$ must lie in the following set,
\begin{equation}\label{set:alpha}
	\Lambda \coloneqq W^{1, \infty}(\R; [0, 1)) \cup \{1\}.
\end{equation}
Avoiding functions which attain values on $[0, 1]$, with $1$ inclusive, is necessary to ensure that the mappings between Eulerian and Lagrangian coordinates are invertible with respect to equivalence classes. See Example \ref{exmp:alphFn1}.

With this setup done, we can define the space of Eulerian coordinates.
\begin{defn}[Set of Eulerian coordinates - $\mD$]\label{def:D}
	Let $\alpha \in \Lambda$. The set $\mD^\alpha$ contains all $Y$, with $ Y = (u,\mu,\nu) $, satisfying the following
	\begin{itemize}
		\item $ u\in E_2 $,   
		\item $ \mu \leq \nu \in \mathcal{M}^+(\R) $,
		\item $\mu_{\mathrm{ac}}\leq  \nu_{\mathrm{ac}} \in \mathcal{M}^+(\R)$,
		\item $ \mu_{\mathrm{ac}} = u_x^2\, dx$, 
		\item $ \mu\left((-\infty, x)\right) \in E_0 $,
		\item $ \nu\left((-\infty, x)\right) \in E_0 $,
		\item If $\alpha \equiv 1, \nu_{ac} = \mu =  u_x^2\, dx, $
		\item If $\alpha \in W^{1,\infty}(\R; [0, 1))$, then $\frac{d\mu}{d\nu}(x) > 0$, and $\frac{d\mu_{\mathrm{ac}}}{d\nu_{\mathrm{ac}}}(x) = 1$ if $u_x(x)<0$, for any $x \in \R$,
	\end{itemize}
	where $\mathcal{M}^+(\R)$ is the set of all finite, positive Radon measures on $\R$. 
	
	The set $\mD$ is defined as
	\[
		\mD \coloneqq
		\left\{ 
			 Y^\alpha \coloneqq (Y, \alpha) \mid 
			 	\alpha \in \Lambda \text{ and } Y \in \mD^\alpha
		\right\} 
		= \bigcup_{\alpha \in \Lambda} \left(\mD^\alpha \times \{\alpha\}\right).
	\]

	Finally, for $M$, $L \geq0$, the subset $\mD_M^{L}$ is given by
	\begin{equation}\label{eqn:DM}
		\mD_M^{L}
		\coloneqq 
		\{
			Y^{\alpha} \in \mD
			\mid
			\mu(\R) \leq M \text{ and } \|\alpha'\|_\infty \leq L
		\}.
	\end{equation}
\end{defn}

Before defining the Lagrangian coordinates, we introduce a new Banach space $B$,
\[
	B \coloneqq E_2 \times E_2 \times E_1 \times E_1, \quad\ \|(f_1, f_2, f_3, f_4)\|_B = \|f_1\|_{E_2} + \|f_2\|_{E_2} + \|f_3\|_{E_1} + \|f_4\|_{E_1}.
\]
\begin{defn}[Set of Lagrangian coordinates - $\mF$]\label{defn:F}
	Let $\alpha \in \Lambda$. The set $\mF^\alpha$ contains all $ X = (y,U,H,V)$ such that $(y-\text{id},U,H,V)\in B$, satisfying
	\begin{itemize}
		\item $y-id, U, H, V \in W^{1, \infty}(\R)$,
		\item $y_\xi, H_\xi \geq 0$, and there exists a constant $c$ such that $0<c<y_\xi + H_\xi$ a.e.,
		\item $y_\xi V_\xi = U_\xi^2$,
		\item $ 0\leq V_\xi \leq H_\xi $ a.e.,
		\item If $\alpha \equiv 1, y_\xi(\xi) = 0$ implies $V_\xi(\xi) = 0$, $y_\xi(\xi)>0$ implies $V_\xi(\xi) = H_\xi(\xi)$ a.e.,
		\item If $0\leq \alpha < 1$, there exists $\kappa: \R \to (0, 1]$ such that $V_\xi(\xi) = \kappa(y(\xi)) H_\xi(\xi)$ a.e., with $\kappa(y(\xi)) = 1$ for $U_\xi(\xi) < 0$.
	\end{itemize}
	
	The set $\mF$ is defined as
	\[
		\mF \coloneqq
		\left\{ 
			 X^\alpha \coloneqq (X, \alpha) 
			 	\mid \alpha \in \Lambda \text{ and } X \in \mF^\alpha
		\right\}
		= \bigcup_{\alpha \in \Lambda} \left(\mF^\alpha \times\{\alpha\}\right).
	\]

	Finally, for $M$, $L \geq 0$, the subset $\mF_M^{L}$ is given by
	\begin{equation}\label{eqn:FM}
		\mF_M^{L} 
		\coloneqq \{X \in \mF \mid \|V\|_\infty \leq M \text{ and } \|\alpha'\|_\infty \leq L\}.
	\end{equation}
\end{defn}

For $\alpha \in \Lambda$, define the set $\mF_0^\alpha$ and $\mF_0$ as
\[
	\mF_0^\alpha \coloneqq \big\{X\in\mF^\alpha \mid y+H=\id \big\},
\]
and
\[
	\mF_0 \coloneqq 
	\big\{X^\alpha\coloneqq(X, \alpha)\in\mF 
		\mid y+H=\id \big\}
	= \bigcup_{\alpha \in \Lambda} \left(\mF_0^\alpha \times \{ \alpha\}\right).
\]
Similar, we set $\mF_{0,M}^{L} = \mF_0 \cap \mF_M^{L}$.

In the general case, where wave breaking can occur, the $\alpha$-dissipative solutions to the Hunter--Saxton equation in Lagrangian coordinates are defined as follows.

\begin{defn}[$\alpha$-Dissipative Solution]\label{def:asolLag} Let $X_0^\alpha=(X_0,\alpha)\in \mF$. We say that $X^\alpha=(X,\alpha)$ is an $\alpha$-dissipative solution with the given initial data $X_0^\alpha$ if $X(t)\in \mF^\alpha$ for all $t\geq 0$ and satisfies
\begin{subequations}\label{eqn:LagSys}
	\begin{align}\label{eqn:LagSys1}
	 y_t(\xi, t) &= U(\xi, t), \\
	 U_t(\xi, t) &= \frac{1}{2} V(\xi, t) 
	 	- \frac{1}{4} V_\infty(t), \\ \label{eqn:LagSys3}
	 H_t(\xi, t) &= 0,\\ \label{eqn:LagSys4}
	 V(\xi, t) &= \int_{-\infty}^\xi V_\xi(\eta, 0 ) (1 - \alpha(y(\eta, \tau(\eta)))\mathds{1}_{\{r \in \R\mid t\geq \tau(r) > 0\}}(\eta))\, d\eta,
	\end{align}
\end{subequations}
with initial data $X(0)=X_0$, 
where $\displaystyle V_\infty(t) = \lim_{\xi \to +\infty}V(\xi,t)$.
\end{defn}

Observe that $\alpha$ is independent of time in the above definition, but is essential since the ODE system \eqref{eqn:LagSys} depends heavily on the choice of $\alpha$. 
Furthermore, note that the derivative $V_\xi$ is in general a discontinuous function in time for particles $\xi \in \R$ experiencing wave breaking.

Existence and uniqueness for the system \eqref{eqn:LagSys} has been shown in \cite{MR3860266}, with the additional fact that the wave breaking time for a particle $\xi \in \R$ is given by
	\begin{equation}\label{def:tau}
		\tau(\xi) = \begin{cases}
			-2\frac{y_\xi(\xi, 0)}{U_\xi(\xi, 0)}, &\mbox\ U_\xi(\xi, 0)<0,\\
			0, &\mbox\ U_\xi(\xi, 0)=0=y_\xi(\xi, 0),\\
			+\infty, &\mbox\ \text{otherwise.}
		\end{cases}
	\end{equation}
We will now introduce some simple estimates that we will make use of later on.

\begin{lem}\label{lem:ODEest}
Consider two $\alpha$-dissipative solutions $X_A^{\alpha_A}$ and $X_B^{\alpha_B}$ with initial data
$X_{0,A}^{\alpha_A}$  and $X_{0,B}^{\alpha_B}$ in $\mF$. Then for each fixed $\xi \in \R$ the following estimates hold
\begin{subequations}
	\begin{align}
		|y_A(\xi, t) - y_B(\xi, t)| &\leq |y_A(\xi, 0) - y_B(\xi, 0)| + \int_0^t |U_A(\xi, s) - U_B(\xi ,s)|\, ds,\\
		|U_A(\xi, t) - U_B(\xi, t)| &\leq |U_A(\xi, 0) - U_B(\xi, 0)| \label{ineq:ODEestU}\\
		&\quad\ + \frac{1}{4}\int_0^t \|V_{A,\xi}(\cdot, s) - V_{B,\xi}(\cdot, s)\|_1\, ds. \nonumber
	\end{align}
\end{subequations}
\end{lem}

\begin{proof}
The first estimate is immediate from the ODE system~\eqref{eqn:LagSys}. We focus on the second. For a fixed $\xi\in \mathbb{R}$,
\begin{align*}
	U_A(\xi, t) - U_B(\xi, t) 
	&= U_A(\xi, 0) - U_B(\xi, 0)\\
	&\quad\ + \int_0^t \bigg( \frac{1}{2}(V_{A}(\xi, s) - V_{B}(\xi, s)) - \frac{1}{4}(V_{A,\infty}(s) - V_{B,\infty}(s)) \bigg)\, ds\\
	&= U_A(\xi, 0) - U_B(\xi, 0) \\
	&\quad\ + \frac{1}{4} \int_0^t \bigg( \int_{-\infty}^{\xi} (V_{A,\xi}(\eta, s) - V_{B,\xi}(\eta, s))\, d\eta\\
	&\hphantom{\quad\ + \frac{1}{4} \int_0^t \bigg( \int_{-\infty}^{\xi}}
	- \int_{\xi}^{\infty} (V_{A,\xi}(\eta, s) - V_{B,\xi}(\eta, s))\, d\eta 
	\bigg)\, ds.
\end{align*}
Hence
\begin{equation*}
	|U_A(\xi, t) - U_B(\xi, t)|
         \leq  |U_A(\xi, 0) - U_B(\xi, 0)| + \frac{1}{4}\int_0^t \|V_{A,\xi}(\cdot, s) - V_{B,\xi}(\cdot, s)\|_1\, ds,
\end{equation*}
as required.
\end{proof}

As a consequence, we have the following corollary.

\begin{cor}\label{cor:lagEsts}
For two $\alpha$-dissipative solutions $X_A^{\alpha_A}$ and $X_B^{\alpha_B}$ with initial data $X_{0,A}^{\alpha_A}$ and $X_{0,B}^{\alpha_B}$ in $\mF$, we have
\begin{subequations}
\begin{align}
	\|y_A(t)-y_B(t)\|_\infty 
	&\leq \|y_A(0) - y_B(0)\|_\infty + \int_0^t \|U_A(s) - U_B(s)\|_\infty\, ds,\\
	\|U_A(t)-U_B(t)\|_\infty 
	&\leq \|U_A(0) - U_B(0)\|_\infty\\
	&\quad\ + \frac{1}{4} \int_0^t \|V_{A,\xi}(s) - V_{B,\xi}(s)\|_1\, ds, \nonumber\\
	\|y_{A,\xi}(t)-y_{B,\xi}(t)\|_2
	&\leq \|y_{A,\xi}(0) - y_{B,\xi}(0)\|_2\\
		&\quad\ + \int_0^t \|U_{A,\xi}(s) - U_{B,\xi}(s)\|_2\, ds, \nonumber\\
	\|U_{A,\xi}(t)-U_{B,\xi}(t)\|_2 
	&\leq \|U_{A,\xi}(0) - U_{B,\xi}(0)\|_2\\ 
	&\quad\ + \hlf \int_0^t \|V_{A,\xi}(s) - V_{B,\xi}(s)\|_2\, ds. \nonumber
\end{align}
\end{subequations}
\end{cor}

\subsection{Mappings between Eulerian and Lagrangian coordinates}
The goal now is to introduce a way of mapping from Eulerian to Lagrangian coordinates and back. These mappings were developed from similar ones for the more complicated Camassa--Holm equation \cite{MR2372478}, and will be central in using a metric in Lagrangian coordinates to define a metric in Eulerian coordinates.
	
\begin{defn}[Mapping $\Lag:\mD\to\mF_0$]\label{map:EultoLag}
	The mapping $\Lag:\mD \to \mF_0$, from Eulerian to Lagrangian coordinates, is defined by
	\begin{equation*}
	\Lag(Y^\alpha)=\Lag((Y,\alpha))=(X, \alpha)=X^\alpha
	\end{equation*}
	with $X=(y,U,H,V)$ given by
	\begin{subequations}\label{map:L}
		\begin{align}
			y (\xi) &= \sup\{ x\in\R \mid x + \nu\big((-\infty ,x)\big) < \xi \},\\
			U(\xi) &= u(y(\xi)),\\
			H(\xi) &= \xi - y(\xi),\\
			V(\xi) &= \int_{-\infty}^\xi H_\xi(\eta) \frac{d\mu}{d\nu}\circ y(\eta)\, d\eta.
		\end{align}
	\end{subequations}
\end{defn}

\begin{defn}[Mapping $\Eul:\mF \to \mD$]\label{map:LagtoEul}
	The mapping $\Eul:\mF\to\mD$, from Lagrangian to Eulerian coordinates, is defined by
	\begin{equation*}
	\Eul(X^\alpha)=\Eul((X,\alpha))=(Y,\alpha)=Y^\alpha
	\end{equation*}
	with $Y=(u,\mu,\nu)$ given by
	\begin{subequations}
		\begin{align}
		u(x) &= U(\xi), \quad\ \text{ for all } \xi\in\R \text{ such that } x=y(\xi),\\
		\mu &= y_{\#}(V_\xi\, d\xi),\\
		\nu &= y_{\#}(H_\xi\, d\xi).
		\end{align}
	\end{subequations}
	Here, we have used the push forward measure for a measurable function
	$f$ and a $\mu$-measurable set $f^{-1}(A)$, i.e.,
	\[
	f_{\#}(\mu)(A) \coloneqq \mu(f^{-1}(A)).
	\]
\end{defn}

The mapping $\hat{L}$ maps four Eulerian coordinates $((u,\mu,\nu),\alpha)$ to five Lagrangian coordinates $((y,U,H,V),\alpha)$.
Hence there is some redundancy here. 
That is to say, a set of Lagrangian coordinates can represent the same Eulerian coordinates. 
This set is an equivalence class, whose elements are related by what is referred to as a ``relabelling''.
\begin{defn}[Relabelling]\label{defn:equivRel}
Let $\mG$ be the group of homeomorphisms
$f:\R\to\R$ satisfying
\begin{equation}\label{proprt:eqf}
	f-\id \in E_2\cap W^{1,\infty}(\R)\quad \text{and} \quad\ f^{-1}-\id \in W^{1,\infty}(\R).
\end{equation}
We define the group action $\circ:\mF\times \mG \to \mF$, called the relabelling of $X^\alpha \in \mF$ by $f$, as
\[
	(X^\alpha,f) \mapsto X^\alpha\circ f=((y\circ f, U\circ f, H\circ f, V\circ f), \alpha).
\]
Hence, one defines the equivalence relation $\sim$ on $\mF$ by
\[
	X^{\alpha_A}_A \sim X^{\alpha_B}_B \text{ if there exists }f\in \mG \text{ such that } X^{\alpha_A}_A = X^{\alpha_B}_B\circ f.
\]
Finally, define the mapping $\Pi:\mF\to \mF_0$, which gives one representative in $\mF_0$ for each equivalence class,
\[
	\Pi(X^\alpha) = X^\alpha\circ (y + H)^{-1}.
\]
\end{defn}	

Under these equivalence classes, the mappings $ \Lag $ and $ \Eul $ are inverses of one another~\cite{MR3860266, MR3573580}.
\begin{lem}\label{lem:Mrel}
Let $ Y^\alpha \in \mD$, and $ \Lag(Y^\alpha) = X^\alpha $. Then, for any $f \in \mG$,
\[
	\Eul(X^\alpha) = Y^\alpha = \Eul(X^\alpha \circ f).
\]
\end{lem}

Further, the relabelling is carried forward in time by the solution,
see~\cite[Proposition 3.7]{MR3860266}.
\begin{lem}\label{lem:Stof} 
Denote by $S_t:\mF \to \mF, X_0^\alpha \mapsto S_t(X_0^\alpha)$ for $ t \in [0, +\infty) $ the solution operator defined in Definition~\ref{def:asolLag} through the ODE system \eqref{eqn:LagSys}.
Then, for any initial data $ X_0^\alpha \in \mF$, and any relabelling function $f \in \mG$,
\[
	S_t(X_0^\alpha \circ f) = S_t(X_0^\alpha) \circ f.
\]
\end{lem}

At this point, we should explore what a solution to the Hunter--Saxton equation can look like.

\begin{exmp}\label{exmp1}
Consider as initial data
\[
	u_0(x) = \begin{cases}
		1+x , &\mbox\ -1 < x \leq 0,\\
		1-x , &\mbox\ 0 < x \leq 1,\\
		0, &\mbox\ \text{otherwise},
	\end{cases}
	\quad\
	\nu_0 = \mu_0 = u_{0,x}^2\, dx,
\]
and $\alpha\in \Lambda$ such that $\alpha(2)=\hlf$. 

The corresponding $\alpha$-dissipative solution is given by
\begin{equation}\label{exmp1:u}
	u(x, t) = \begin{cases}
		\begin{cases}
			-\hlf t, &\mbox\ x \leq -\frac{1}{4}t^2 - 1,\\
			\frac{2-t+2x }{t + 2}, &\mbox\ -\frac{1}{4}t^2 - 1 < x \leq t,\\
			\frac{-2-t+2x}{t-2}, &\mbox\ t < x \leq \frac{1}{4}t^2 + 1,\\
			\hlf t, &\mbox\ \frac{1}{4}t^2 + 1 < x,
		\end{cases} &\mbox\ 0 \leq t < 2,\\
		\begin{cases}
		- \frac{1}{4}	-\frac{3}{8} t, &\mbox\ x \leq -\frac{3}{16}t^2 - \frac{1}{4}	t - \frac{3}{4},\\
			\frac{2-t+4x}{2(t+2)}, &\mbox\ -\frac{3}{16}t^2 - \frac{1}{4}t - \frac{3}{4} < x \leq \frac{1}{16}t^2 + \frac{3}{4}t + \frac{1}{4},\\
			\frac{-2-t+2x}{t-2}, &\mbox\ \frac{1}{16}t^2 + \frac{3}{4}t + \frac{1}{4} < x \leq \frac{3}{16}t^2 + \frac{1}{4}t + \frac{3}{4},\\
			\frac{1}{4} + \frac{3}{8} t, &\mbox\ \frac{3}{16}t^2 + \frac{1}{4}t + \frac{3}{4} < x,
		\end{cases} &\mbox\ 2<t.
	\end{cases}
\end{equation}
with
\[
	\mu(t) = u_x^2(t)\, dx + \hlf \delta_{2}\mathbbm{1}_{\{t=2\}}(t).
\]
See Figure \ref{fig:exmp1} for plots of $u$ at different times.

\begin{figure}
\begin{center}
	\includegraphics[scale=0.6]{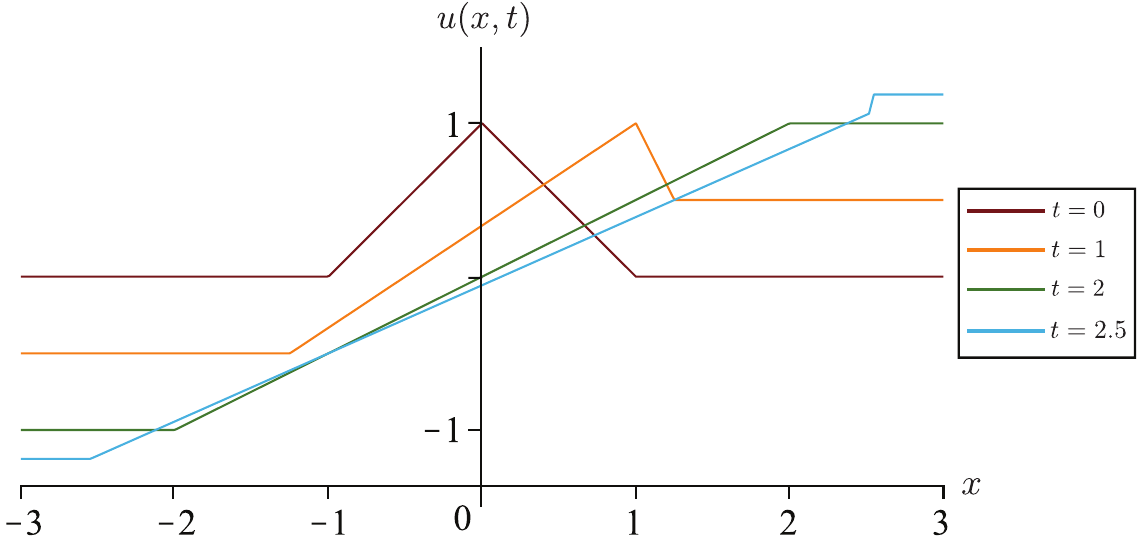}
\end{center}
\caption{Plots of $u$, as given by \eqref{exmp1:u}, at different times.} 
\label{fig:exmp1}
\end{figure}

Note that the third interval shrinks into the single point $x=2$ as $t\to 2$, and the derivative $u_x \to -\infty$ as $t\to 2$. Of course we retain that $u$ is a distributional solution regardless of the value of $u$ at this point. However, $u(\cdot,t)\in E_2$ and therefore $u(2,2) = 1$. 

Furthermore, note that all $\alpha\in \Lambda$, which satisfy $\alpha(2)=\frac12$, yield the same $\alpha$-dissipative solution. This is due to wave breaking occurring once at $(t,x)=(2,2)$ for all of these $\alpha$-dissipative solutions. As a consequence, it is vital to consider $Y^\alpha$ instead of $Y$, when constructing our metric.
 
\end{exmp}

With our notation in place, we introduce the definition of an $\alpha$-dissipative solution for \eqref{eqn:HS}.

\begin{defn}[$\alpha$-Dissipative Solution]\label{def:euleralpha}
Let $ Y_0^\alpha=(Y_0, \alpha) = ((u_0, \mu_0, \nu_0), \alpha) \in \mD $. We say $Y^\alpha=(Y,\alpha) =( (u, \mu, \nu), \alpha)$ is a weak solution with the given initial data $Y_0^\alpha$ if the following conditions are satisfied,
\begin{subequations}
\begin{alignat}{2}
	u &\in C^{0,\hlf}(\R \times [0, T]), \quad\ &&\text{ for any } T \geq 0,\\
	\nu &\in C_{weak*}([0, +\infty); \mathcal{M}^+(\R)),\\
	Y(t) &\in \mD^\alpha, &&\text{ for any } t \in [0, +\infty) ,\\
	Y(0) &= Y_0,\\
	\nu(t)(\R) &= \nu_0(\R), && \text{ for any } t \in [0, +\infty).
\end{alignat}
\end{subequations}
Further, $u$ must satisfy \eqref{eqn:HS} in the distributional sense, that is, for any test function $ \varphi \in \Cc(\R \times [0, +\infty)) $ with $ \varphi(x, 0) = \varphi_0(x) $,
\begin{equation}
	\int_0^{+\infty} \int_\R 
	\bigg[ 
		u \varphi_t 
		+ \frac{1}{2} u^2 \varphi_x 
		+ \frac{1}{4}
		\bigg( 
			\int_{-\infty}^x\, d\mu - \int_x^{+\infty}\, d\mu 
		\bigg) 
		\varphi
	\bigg]\, dx\, dt
	= -\int_\R u_0 \varphi_0\, dx,
\end{equation}
and $ \mu $ must satisfy \begin{equation}\label{ineq:nonConsMu}
	\int_0^{+\infty} \int_\R [\phi_t + u \phi_x]\, d\mu(t)\, dt \geq - \int_\R \phi_0\, d\mu_0,
\end{equation}
for every non-negative test function 
\begin{equation*}
	\phi \in \Cc(\R \times [0, +\infty); [0, +\infty)),
	\quad \text{with} \quad
	\phi(x, 0) = \phi_0(x).
\end{equation*}

Finally, we say that $Y^\alpha$ is an $\alpha$-dissipative solution if $Y^\alpha$ is a weak solution and if for each $ t \in [0, +\infty) $,
\begin{subequations}
\begin{alignat}{2}
	d \mu(t) &= d\mu^-_{ac}(t) + &&(1-\alpha(x))d\mu_s^-(t),\\
	\mu(s) &\overset{\ast}{\rightharpoonup} \mu(t), &&\text{as } s \downarrow t,\\
	\mu(s) &\overset{\ast}{\rightharpoonup} \mu^-(t), &&\text{as } s \uparrow t.
\end{alignat}
\end{subequations}
\end{defn}

\begin{note}
	If $Y^\alpha(t)$ is a conservative solution, then \eqref{ineq:nonConsMu} will be an equality.
\end{note}

Bringing everything together, define $ T_t:\mD\to \mD$ for $t\in [0,+\infty)$  as
\[
	T_t Y_0^\alpha =( \Eul \circ S_t \circ \Lag)Y_0^\alpha.
\]
Then $T_t$ associates to
 each initial data $Y_0^\alpha=(Y_0, \alpha) \in\mD $ an $\alpha$-dissipative solution in the sense of Definition~\ref{def:euleralpha}. The proof can be found in \cite[Theorem 3.14]{MR3860266}. Henceforth when referring to $\alpha$-dissipative solutions in Eulerian coordinates, we refer to the solutions given by $T_t$.

Finally, it is important to observe that $u$ and $\mu$ are independent of $\nu$ and therefore it is possible to introduce equivalence classes in Eulerian coordinates.

\begin{lem}
Let $Y_A^{\alpha_A}$ and $Y_B^{\alpha_B}$ be two $\alpha$-dissipative solutions with initial data $Y_{A,0}^{\alpha_A}$ and $Y_{0,B}^{\alpha_B}$ in $\mD$.
If 
\begin{equation}\label{Uavh:nu}
u_{0,A}=u_{0,B}, \quad \mu_{0,A}=\mu_{0,B} \quad \text{ and } \quad \alpha_A=\alpha_B,
\end{equation}
then 
\begin{equation*}
u_A(\cdot,t)=u_B(\cdot,t) \quad \text{ and } \quad  \mu_A(t)=\mu_B(t) \quad \text{ for all }t\geq 0.
\end{equation*}
\end{lem}

\begin{proof}

Without loss of generality assume that $\mu_{0,A}=\nu_{0,A}$. 

Introduce $X_{0,i}^{\alpha_i}=( (y_{0,i}, U_{0,i}, V_{0,i}, H_{0,i}), \alpha_i)=\hat{L}(Y_{0,i}^{\alpha_i})$ for $i=A$, $B$. We claim there exists an increasing and Lipschitz continuous function $g$ such that 
\begin{equation}\label{claim}
(y_{0,A}\circ g, U_{0,A}\circ g, V_{0,A}\circ g)=(y_{0,B}, U_{0,B}, V_{0,B}).
\end{equation}
By assumption $V_{0,A}(\xi)=H_{0,A}(\xi)$ for all $\xi\in \mathbb{R}$ and hence 
\begin{equation*}
y_{0,A}(\xi)+V_{0,A}(\xi)=\xi \quad  \text{ for all }\xi\in \mathbb{R}. 
\end{equation*}
For $V_{0,B}(\xi)$, on the other hand, we have that there exists a function $\kappa: \mathbb{R}\to [0,1]$ such that 
\begin{equation*}
V_{0,B,\xi}(\xi)= \kappa(y_{0,B}(\xi))H_{0,B,\xi}(\xi) \quad \text{ for all }\xi \in \mathbb{R},
\end{equation*}
which implies that 
\begin{align*}
y_{0,B}(\xi)+V_{0,B}(\xi)&= y_{0,B}(\xi)+ H_{0,B}(\xi)+V_{0,B}(\xi)-H_{0,B}(\xi)\\
& = \xi - \int_{-\infty}^\xi (1-\kappa(y_{0,B}(\eta)))H_{0,B,\xi}(\eta) d\eta,
\end{align*}
where the function on the right hand side is increasing and Lipschitz continuous with Lipschitz constant at most one. 
Introduce 
\begin{equation*}
g(\xi)=\xi - \int_{-\infty}^\xi (1-\kappa(y_{0,B}(\eta)))H_{0,B,\xi}(\eta) d\eta,
\end{equation*}
then 
\begin{equation}\label{get:contr}
y_{0,B}(\xi)+V_{0,B}(\xi)=g(\xi)=y_{0,A}(g(\xi))+V_{0,A}(g(\xi)) \quad \text{ for all } \xi\in \mathbb{R}.
\end{equation}
Next, we establish that $y_{0,A}(g(\xi))= y_{0,B}(\xi)$ for all $\xi \in \mathbb{R}$. Assume the opposite, i.e., there exists $\bar \xi \in \mathbb{R}$ such that $y_{0,A}(g(\bar \xi))\not = y_{0, B}(\bar \xi)$ and without loss of generality we assume that 
\begin{equation}\label{cont:ass}
y_{0,A}(g(\bar \xi))<y_{0, B}(\bar \xi). 
\end{equation}
Since \eqref{map:L} implies for $i=A$, $B$,
\begin{equation*}
\mu_{0,i}((-\infty, y_{0,i}(\xi)))\leq V_{0,i}(\xi)\leq \mu_{0,i}((-\infty, y_{0,i}(\xi)]) \quad \text{ for all }\xi\in \mathbb{R},
\end{equation*}
we have, recalling \eqref{Uavh:nu} and using \eqref{get:contr},
\begin{align*}
\mu_{0,A}((-\infty, y_{0,B}(\bar \xi)))& =\mu_{0,B}((-\infty,y_{0,B}(\bar \xi))) \leq V_{0,B}(\bar \xi)\\
& <V_{0,A}(g(\bar \xi))\leq \mu_{0,A}((-\infty, y_{0,A}(g(\bar \xi))])
\end{align*}
Since this is only possible if $y_{0,B}(\bar \xi))\leq y_{0,A}(g(\bar \xi))$, we end up with a contradiction to \eqref{cont:ass}. Thus $y_{0,A}\circ g= y_{0,B}$ and, by Definition~\ref{map:EultoLag}, $V_{0,A}\circ g=V_{0,B}$ and 
\begin{equation}
	U_{0,A} \circ g = u \circ y_{0,A} \circ g = u \circ y_{0,B} = U_{0,B},
\end{equation}
which finishes the proof of \eqref{claim}.

Next, we show
\begin{equation}\label{claim2}
(y_A,U_A,V_A)(g(\xi),t)=(y_B,U_B,V_B)(\xi,t)\quad \text{ for all }\xi \in \mathbb{R}\text{ and }t\geq 0.
\end{equation}
Therefore, observe that the system of ordinary differential equations given by \eqref{eqn:LagSys1}--\eqref{eqn:LagSys3} is a closed system for $(y,U,V)$ and hence $H$ does not influence the time evolution of $(y,U,V)$. Furthermore, recalling that $\alpha_A=\alpha_B$ and repeating the argument of \cite[Proposition 3.7]{MR3860266}, one finds \eqref{claim2}.

Finally, we can apply the mapping $M$ to go back to Eulerian coordinates as follows. Let $(x,t)\in \mathbb{R}\times \mathbb{R}^+$, then there exists $\xi \in \mathbb{R}$ such that 
\begin{equation*}
y_A(g(\xi),t)=x=y_B(\xi,t)
\end{equation*}
and hence 
\begin{equation*}
u_A(x,t)=U_A(g(\xi),t)=U_B(\xi,t)=u_B(x,t).
\end{equation*}
Furthermore, let $\bar\xi=\sup \{\eta\mid y_B(\eta,t)<x\}$, then $g(\bar \xi)=\sup \{\eta\mid y_A(\eta,t)<x\}$ and therefore
\begin{equation*}
\mu_A((-\infty,x),t)=\int_{-\infty}^{g(\bar\xi)}V_{A,\xi}(\eta,t)d \eta= V_A(g(\bar \xi),t)=V_B(\bar \xi,t)=\mu_B((-\infty,x),t).
\end{equation*}
\end{proof}

We can now define a new set that will contain triplets $Z^\alpha=(Z,\alpha)=((u,\mu),\alpha)$ that form the solution to
\eqref{eqn:HS}.

\begin{defn}[Equivalence classes in $\mD$]\label{def:equiv:euler}
The set $\mD_0$ contains all $Z^\alpha=(Z,\alpha)=((u,\mu), \alpha)\in E_2\times \mathcal{M}^+(\R)\times \Lambda$ satisfying
\begin{itemize}
\item $ \mu_{ac} = u_x^2\, dx $,
	\item $ \mu = u_x^2\, dx$ if $\alpha = 1$,
	\item $\mu((-\infty, x))\in E_0$.
\end{itemize}
Then, for each $Z^\alpha=(Z,\alpha)=((u, \mu),\alpha) \in \mD_{0}$ we define the set
\[
	\mathcal{V}(Z^\alpha) \coloneqq \{ \nu \in \mathcal{M}^+(\R) \mid ((Z, \nu),\alpha) \in \mD\},
\]
i.e. the equivalence class of all $\nu$ related by having the same $ Z^\alpha=((u, \mu),\alpha) $.

Finally, for $M$, $L\geq 0$, define $\mD_{0, M}^{L}$ by
\[
	\mD_{0, M}^{L}
	\coloneqq \left\{
		Z^\alpha\in \mD_0\mid 
	 \mu(\R) \leq M \text{ and } \|\alpha'\|_\infty \leq L\right \}.
\]
\end{defn}

\begin{note}
$\mD$ can be written as
\[
	\mD = \{ ((u, \mu, \nu), \alpha) 
	\mid 
	 ((u,\mu),\alpha) \in \mD_{0} \text{ and } \nu \in \mathcal{V}((u, \mu),\alpha) \}.
\]

\end{note}

\begin{note}
Under the present setting, uniqueness of fully dissipative solutions has been established in~\cite{MR2796054}. For the conservative case, uniqueness was shown in~\cite{MR4400175}.
\end{note}

\section{A metric in Lagrangian coordinates}
Our first goal is to introduce a metric in the space of Lagrangian coordinates that is Lipschitz stable with respect to initial Lagrangian coordinates in the sense of equivalence classes.

We begin our approach by introducing a semi-metric, 
i.e. dropping the triangle inequality requirement,
on the set of Lagrangian coordinates.
The most important condition of this mapping is that it is Lipschitz
continuous with respect to the initial data in $\mathcal{F}$.
We will then, in the next section, use this semi-metric to define a metric on the space
of equivalence classes in Lagrangian coordinates, ensuring that 
Lagrangian coordinates representing the same Eulerian coordinates
have a distance of zero.

We introduce important sets that our construction will take advantage of.

\subsection{Some important sets}
For two $\alpha$-dissipative solutions 
$X_i^{\alpha_i}$,  $X_j^{\alpha_j}$,
with labels $i$ and $j$, define the sets
\begin{subequations}\label{defn:normSets}
\begin{align}
	&\mA_i(t) = \mA(X_i^{\alpha_i}; t) = \left\{ \xi \in \R \mid U_{i, \xi}(\xi, t) \geq 0 \right\},\\
	&\mA_{i,j}(t) = \mA_i(t) \cap \mA_j(t),\\
	&\mB_{i, j}(t) = \mB(X_i^{\alpha_i}, X_j^{\alpha_j}; t) = \left\{ \xi \in \R \mid t < \tau_i(\xi) = \tau_j(\xi) < \infty \right\},\\
	&\Omega_{i,j}(t) = \Omega(X_i^{\alpha_i}, X_j^{\alpha_j}; t) = \mA_{i,j}(t) \cup \mB_{i,j}(t).
\end{align}
\end{subequations}
Should $X_i, X_j$ be just elements of $\mF$ (with no time dependence), take $t=0$ in the definitions, and naturally these will no longer be dependent on time.

We can describe the contents of these sets as follows
\begin{itemize}
	\item $\mA_{i,j}(t)$ contains the particles $\xi$ for which no wave breaking will occur for both solutions at any point in the future. 
	\item $ \mB_{i, j}(t) $ contains the $\xi$ for which wave breaking will occur in both solutions at the same time in the future.
	\item $ \Omega_{i,j}^c(t) $ contains everything else, i.e. particles for which wave breaking occurs at different times in the future, or for which only one of the two will break.
\end{itemize}
Importantly, these three sets form a disjoint union of the entire real line and are independent of the choice of $\alpha$.

Furthermore, elements $\xi$ of the sets $\mB_{i,j}(t)$ and $\Omega_{i,j}(t)$
remain in their respective set until both have broken and $\xi$ enters $\mA_{i,j}(t)$.

A natural question is ``how do these sets change after a relabelling of the Lagrangian coordinates?''. To begin answering this question, we introduce the following notation:

For $X_i^{\alpha_i}$ and $X_j^{\alpha_j}$ in $\mF$, and $f, h \in \mG$ define
\begin{subequations}\label{defn:normSetsLabel}
\begin{align}
	&\mA_i^f(t) = \mA(X_i^{\alpha_i}\circ f; t)\\
	&\mA_{i,j}^{f,h}(t) = \mA_i^f(t) \cap \mA_j^h(t),\\
	&\mB_{i,j}^{f,h}(t) = \mB(X_i^{\alpha_i} \circ f, X_j^{\alpha_j}\circ h; t)\\
	&\Omega_{i,j}^{f,h}(t) = \mA_{i,j}^{f,h}(t) \cup \mB_{i,j}^{f,h}(t).
\end{align}
\end{subequations}
If $f$ and $h$ are the identity functions, then this notation collapses back to that in $\eqref{defn:normSets}$.

Consider two functions $f$ and $h$ in $\mG$, the set of relabelling functions, as given by Definition \ref{defn:equivRel}. Such functions are continuous and strictly monotonically increasing, i.e. $ f_\xi(\xi) > 0 $, almost everywhere, cf. \cite[Lemma 3.2]{MR2372478}.

Let $X^\alpha \in \mF$. Then
\begin{equation}\label{eqn:relA}
\begin{split}
	\mA^f
	&= \{ \xi\in\R \mid (U \circ f)_\xi(\xi) \geq 0 \}\\
	&= \{ \xi\in\R \mid (U_{\xi} \circ f)(\xi)f_\xi(\xi) \geq 0 \} \\
	&= \{ \xi\in\R \mid (U_{\xi} \circ f)(\xi) \geq 0 \}\\
	&= \{ \xi \in \R \mid f(\xi) \in \mA \}
	= f^{-1}(\mA),
\end{split}
\end{equation}
or equivalently, $\mA = f(\mA^f)$.

Inspired by the previous calculation, we look at the other sets. We have, as $f$ is bijective, for $X_A^{\alpha_A}$ and $X_B^{\alpha_B} \in \mF$,
\begin{equation}
\begin{split}
\label{eqn:SetRelAfh}
	f(\mA_{A,B}^{f,h})
	= f(\mA_A^f) \cap f(\mA_B^h)
	= \mA_A \cap \mA_B^{h \circ f^{-1}}
	= \mA_{A,B}^{\id, h\circ f^{-1}}.
\end{split}
\end{equation}

We also have a relation for the breaking times after relabelling.
Once again take $X_A^{\alpha_A}$ and $ X_B^{\alpha_B} \in \mF$, and suppose that $f(\eta) \in \mA_A^c$.
Defining temporarily $X_C ^{\alpha_C}= X_A^{\alpha_A} \circ f$, the wave breaking time after relabelling is given by
\[
	\tau_C(\eta)
	= -2\frac{(y_{A}\circ f)_\eta(\eta)}{(U_A\circ f)_\eta(\eta)}
	= -2 \frac{y_{A,\xi}(f(\eta))f_\eta(\eta)}{U_{A,\xi}(f(\eta))f_\eta(\eta)}
	= \tau_A(f(\eta)) \text{ a.e.},
\]
which gives us
\begin{equation}
\begin{split}\label{eqn:SetRelBfh}
	f(\mB_{A,B}^{f,h})
	&= \{ f(\xi) \mid \xi \in \R \text{ and } 0 < \tau_A(f(\xi)) = \tau_B(h(\xi)) < +\infty \}\\
	&= \{ \xi \in \R \mid 0 < \tau_A(\xi) = \tau_B((h\circ f^{-1})(\xi)) < +\infty \}
	= \mB_{A,B}^{\id, h\circ f^{-1}}.
\end{split}
\end{equation}
This has the immediate consequence
\begin{equation}\label{eqn:SetRelOmfh}
	f(\Omega_{A,B}^{f,h,c}(t)) = \Omega_{A,B}^{\id, h\circ f^{-1}, c}(t).
\end{equation}

\subsection{Construction of a semi-metric for Lagrangian coordinates}
We now begin the first step of the construction of our metric, measuring the distance between two $\alpha$-dissipative solutions, where $\alpha \in \Lambda$.

We cannot simply use a metric based on the norms of the Banach space $E$.
This is a consequence of the discontinuities in time of the derivatives $V_\xi$.
For two solutions $X_A^{\alpha_A}$ and $X_B^{\alpha_B}$, the difference $\|V_{A,\xi}(t) - V_{B,\xi}(t)\|_1$
can increase in time and in particular, it can have a jump of positive height.

To resolve this issue, we introduce a new function $G_{A,B}(\xi,t)$ that will decrease in time, and only drops can occur.

Let $X_A^{\alpha_A}$, $X_B^{\alpha_B}$ be two $\alpha$-dissipative solutions.
The following functions will all contribute to the function $G_{A,B}(\xi,t)$.
\begin{equation}\label{eqn:g}
	g_{A,B}(\xi, t)=g(X_A^{\alpha_A}, X_B^{\alpha_B})(\xi,t) = |V_{A,\xi}(\xi, t)  - V_{B,\xi}(\xi, t) |,
\end{equation}

\begin{equation}\label{eqn:ghat}
\begin{split}
	\hat{g}_{A,B}(\xi, t)
	&= \hat{g}(X_A^{\alpha_A}, X_B^{\alpha_B})(\xi,t)\\
	 &= |V_{A,\xi}(\xi, t)  - V_{B,\xi}(\xi, t) | + \|\alpha_A - \alpha_B\|_\infty (V_{A,\xi} \wedge V_{B,\xi})(\xi, t)\\
	 &\quad\ + \|\alpha_{A,B}'\|_\infty (V_{A,\xi} \wedge V_{B,\xi})(\xi, t) \bigg(
	 |y_A(\xi, t) - y_B(\xi, t)|\\
	 &\hphantom{\quad\quad \ + \|\alpha'\|_\infty (V_{A,\xi} \wedge V_{B,\xi})(\xi, t) \bigg(}
	  +|U_A(\xi, t) - U_B(\xi, t)|
	 \bigg),
\end{split}
\end{equation}

\begin{equation}\label{eqn:gbar}
\begin{split}
	\bar{g}_{A,B}(\xi, t)
	&= \bar{g}(X_A^{\alpha_A}, X_B^{\alpha_B})(\xi,t)\\
	&= |V_{A,\xi}(\xi, t)  - V_{B,\xi}(\xi, t) |\\
	&\quad\ + \big(V_{A,\xi} \wedge V_{B,\xi} \big) (\xi, t)(\alpha_A(\xi)\bbo_{\mA_A^c(t)}(\xi) +  \alpha_B(\xi)\bbo_{\mA_B^c(t)}(\xi))\\
	&\quad\ + \|\alpha_{A,B}'\|_\infty (V_{A,\xi} \wedge V_{B,\xi})(\xi, t) \\
	&\hphantom{ \quad\ + \|\alpha_{A,B}'\|_\infty }
		\times
		\bigg(
			|y_A(\xi, t) - \id(\xi)|\bbo_{\mA_A^c(t)}(\xi)\\  
		&\hphantom{ \quad\ + \|\alpha_{A,B}'\|_\infty \times \bigg(}
			+ |y_B(\xi, t) - \id(\xi)|\bbo_{\mA_B^c(t)}(\xi)\\
		&\hphantom{ \quad\ + \|\alpha_{A,B}'\|_\infty \times \bigg(}
			+|U_A(\xi, t)|(\bbo_{\mA_A^c(t)}(\xi) + \bbo_{\mA_B^c(t)}(\xi))\\
		&\hphantom{ \quad\ + \|\alpha_{A,B}'\|_\infty \times \bigg(}
			+ |U_B(\xi, t)|(\bbo_{\mA_A^c(t)}(\xi) + \bbo_{\mA_B^c(t)}(\xi))
		\bigg),
\end{split}
\end{equation}
where
\[
	\alpha_{A,B}' = \alpha_A' \vee \alpha_B'.
\]
Here we use a shorthand notation for the minimum and the maximum. For $a, b \in \R$,
\[
	a \wedge b = \min\{a, b\}\quad  \text{ and }\quad  a \vee b = \max\{a, b\}.
\]

\begin{prop}\label{prop:G}
	Let $X_A^{\alpha_A}$ and $X_B^{\alpha_B}$ be two $\alpha$-dissipative solutions
	with initial data $X_{0,A}^{\alpha_A}$ 
	and $X_{0,B}^{\alpha_B}$ in  $\mF$.
	Define
	\begin{equation}\label{eqn:G}
		\begin{split}
			G_{A,B}(\xi, t) 
		&= G(X_A^{\alpha_A}, X_B^{\alpha_B})(\xi, t)\\
		&= g_{A,B}(\xi, t)\bbo_{\mA_{A,B}(t)}(\xi) + \hat{g}_{A,B}(\xi, t)\bbo_{\mB_{A,B}(t)}(\xi) + \bar{g}_{A,B}(\xi, t)\bbo_{\Omega_{A,B}^c(t)}(\xi) \\
		&\quad\ + \frac{1}{4}\|\alpha_{A,B}'\|_\infty (V_{A,\xi} \wedge V_{B,\xi} )(\xi, t)\\
		&\hphantom{\quad\ + \frac{1}{4}}
		\times (\|V_{A,\xi}(\cdot,t)\|_1 + \|V_{B,\xi}(\cdot,t)\|_1 + 1)\\
		&\hphantom{\quad\ + \frac{1}{4}}
		\times (\bbo_{\mA_A^c(t)}(\xi) + \bbo_{\mA_B^c(t)}(\xi))\bbo_{\mB_{A,B}^c(t)}(\xi)
		\end{split}
	\end{equation}
and let 
\begin{equation}\label{def:MAB}
	M_{A,B} =
	\max(\|V_A(\cdot,0)\|_\infty, \|V_B(\cdot,0)\|_\infty)=	\max(\sup_{t\geq 0} \|V_A(\cdot,t)\|_\infty, \sup_{t\geq 0}\|V_B(\cdot,t)\|_\infty).
\end{equation} 

Then 
\begin{equation}\label{eqn:diffVleqG}
	\|(V_{A,\xi} - V_{B,\xi})(\cdot,t)\|_{i} \leq \|G_{A,B}(\cdot, t)\|_{i} \quad\ \text{ for } i = 1,2,
\end{equation}
and $G_{A,B}$ is a decreasing function over breaking times, 
i.e. 
\begin{equation*}
	G_{A,B}(\xi, \tau(\xi)) \leq \lim_{t \uparrow \tau(\xi)}G_{A,B}(\xi,t).
\end{equation*}

Furthermore, for $X_{0,A}^{\alpha_A}\in \mF_0$ 
	and $X_{0,B}^{\alpha_B}\in\mF$,
\begin{equation}\label{prop:ineq:FL1}
	\|G_{A,B}(\cdot, t)\|_1 \leq \|G_{A,B}(\cdot, 0)\|_1 + \int_0^t (\|G_{A,B}(\cdot,s)\|_1 + \frac14 M_{A,B} \|\alpha_{A,B}'\|_\infty \|G_{A,B}(\cdot,s)\|_1)\, ds,
\end{equation}
and
\begin{equation}\label{prop:ineq:FL2}
	\|G_{A,B}(\cdot,t)\|_2 \leq \|G_{A,B}(\cdot,0)\|_2 + \int_0^t (\|G_{A,B}(\cdot,s)\|_2 +\frac14 \sqrt{M_{A,B}} \|\alpha_{A,B}'\|_\infty \|G_{A,B}(\cdot,s)\|_1)\, d{s}.
\end{equation}
\end{prop}

\begin{proof}
Relationship \eqref{eqn:diffVleqG} is an immediate consequence of the definition of $G_{A,B}$.

We have tactically constructed $G_{A,B}$ such that it can be split into four parts. The first three are defined on disjoint sets whose union is the entire real line, and the final term is necessary in order to obtain \eqref{prop:ineq:FL1} and \eqref{prop:ineq:FL2}.

For a function $h:\R \to \R$ we use the notation, $h(t-) \coloneqq \lim_{s\uparrow t} h(s)$. Furthermore, we drop the $\xi$ for ease of readability, in the following computations.

We begin by demonstrating that $G_{A,B}$ decreases over breaking times $\tau(\xi)$. 

Consider $ \xi \in \mA_{A,B}(t) $ for all time.
These particles do not experience wave breaking, thus the energy at these points is retained, and hence
\begin{equation*}
	 g_{A,B}(t) = |V_{A,\xi}(t)-V_{B,\xi}(t)| 
\end{equation*}
is constant.
For other values of $\xi$ things are not so simple.

For $\xi \in \mB_{A,B}(0)$, at time $\tau(\xi)$, we have
\begin{align*}
	g_{A,B}(\tau) 
	&= |V_{A,\xi}(\tau) - V_{B,\xi}(\tau)|\\
	&= |V_{A,\xi}(\tau-)(1-\alpha_A(y_A(\tau-))) - V_{B,\xi}(\tau-)(1-\alpha_B(y_B(\tau-)))|\\
	&\leq |V_{A,\xi}(\tau-) - V_{B,\xi}(\tau-)|(1-\alpha_A(y_A(\tau-))) \\
	&\quad\ + |\alpha_A(y_A(\tau-)) - \alpha_B(y_B(\tau-))|V_{B,\xi}(\tau-)\\
	&\text{or}\\
	&\leq |V_{A,\xi}(\tau-) - V_{B,\xi}(\tau-)|(1-\alpha_B(y_B(\tau-))) \\
	&\quad\ + |\alpha_A(y_A(\tau-)) - \alpha_B(y_B(\tau-))|V_{A,\xi}(\tau-).
\end{align*}
Using that, for any $t \in [0, +\infty)$,
\begin{align*}
	|\alpha_A(y_A(t)) - \alpha_B(y_B(t))| 
	&\leq |\alpha_A(y_A(t)) - \alpha_B(y_A(t))| + |\alpha_B(y_A(t)) - \alpha_B(y_B(t))| \\
	&\leq \|\alpha_A - \alpha_B\|_\infty + \|\alpha_B'\|_\infty |y_A(t) - y_B(t)|
\end{align*}
and similarly
\[
	|\alpha_A(y_A(t)) - \alpha_B(y_B(t))| 
	\leq \|\alpha_A - \alpha_B\|_\infty + \|\alpha_A'\|_\infty |y_A(t) - y_B(t)|
\]
we find that
\begin{align*}
	g_{A,B}(\tau) 
	&\leq |V_{A,\xi}(\tau-) - V_{B,\xi}(\tau-)| + |\alpha_A(y_A(\tau-)) - \alpha_B(y_B(\tau-))|\big(V_{A,\xi}(\tau-) \wedge V_{B,\xi}(\tau-)\big)\\
	&\leq |V_{A,\xi}(\tau-) - V_{B,\xi}(\tau-)| 
	+ \|\alpha_A - \alpha_B\|_\infty \big(V_{A,\xi}(\tau-) \wedge V_{B,\xi}(\tau-)\big)\\
	&\quad\ + \|\alpha_{A,B}'\|_\infty|y_A(\tau-) - y_B(\tau-)|\big(V_{A,\xi}(\tau-) \wedge V_{B,\xi}(\tau-)\big)\\
	&\leq \hat{g}_{A,B}(\tau-).
\end{align*}

For $\xi \in \Omega_{A,B}^c(0)$, we consider two possibilities. First, we can have one solution breaking at time $\tau(\xi)$, and the other never breaking. Suppose $X_A^{\alpha_A}$ breaks at $\tau_A(\xi)$, then
\begin{align*}
	g_{A,B}(\tau_A) 
	&= |V_{A,\xi}(\tau_A) - V_{B,\xi}(\tau_A)|\\
	&= |V_{A,\xi}(\tau_A-)(1-\alpha_A(y_A(\tau_A-))) - V_{B,\xi}(\tau_A-)|\\
	&\leq |V_{A,\xi}(\tau_A-) - V_{B,\xi}(\tau_A-)| + \alpha_A(y_A(\tau_A-))(V_{A,\xi}(\tau_A-) \wedge V_{B,\xi}(\tau_A-))\\
	&\leq |V_{A,\xi}(\tau_A-) - V_{B,\xi}(\tau_A-)|\\
		&\quad\ + (\alpha_A(y_A(\tau_A-)) - \alpha_A(\id)
		+ \alpha_A(\id)) (V_{A,\xi}(\tau_A-) \wedge V_{B,\xi}(\tau_A-))\\
	&\leq |V_{A,\xi}(\tau_A-) - V_{B,\xi}(\tau_A-)| \\
		&\quad\ + (\|\alpha_A'\|_\infty |y_A(\tau_A-) - \id| + \alpha_A(\id)) (V_{A,\xi}(\tau_A-)\wedge V_{B,\xi}(\tau_A-))\\
	&\leq \bar{g}_{A,B}(\tau_A-).
\end{align*}

The last case is where both break at different times. Suppose $X_A^{\alpha_A}$ breaks first, and $X_B^{\alpha_B}$ second. At time $\tau_B$, we can use the previous result, hence 
\begin{equation*}
g_{A,B}(\tau_B) \leq \bar{g}_{A,B}(\tau_B -). 
\end{equation*}
At time $ \tau_A$, we have
\begin{align*}
	\bar{g}_{A,B}(\tau_A)
	&= |(1-\alpha_A(y_A(\tau_A-)))V_{A,\xi}(\tau_A-) - V_{B,\xi}(\tau_A-)|\\
	&\quad\ + \alpha_B(\id) ((1-\alpha_A(y_A(\tau_A-)))V_{A,\xi}(\tau_A-) \wedge V_{B,\xi}(\tau_A-))\\
	&\quad\ + \|\alpha_{A,B}'\|_\infty 
		\bigg(
			|y_B(\tau_A-) - \id| + |U_A(\tau_A-)| + |U_B(\tau_A-)|
		\bigg)\\
		&\hphantom{\quad\ + \|\alpha'\|_\infty} 
		\times ((1-\alpha_A(y_A(\tau_A-)))V_{A,\xi}(\tau_A-) \wedge V_{B,\xi}(\tau_A-))\\
	&\leq |V_{A,\xi}(\tau_A-) - V_{B,\xi}(\tau_A-)| 
			+ \alpha_B(\id) (V_{A,\xi}(\tau_A-) \wedge V_{B,\xi}(\tau_A-))\\
	&\quad\ + (\alpha_A(y_A(\tau_A-)) - \alpha_A(\id) + \alpha_A(\id))(V_{A,\xi}(\tau_A-) \wedge V_{B,\xi}(\tau_A-))\\
		&\quad\ + \|\alpha_{A,B}'\|_\infty 
			\bigg(
				|y_B(\tau_A-) - \id| + |U_A(\tau_A-)| + |U_B(\tau_A-)|
			\bigg) (V_{A,\xi}(\tau_A-) \wedge V_{B,\xi}(\tau_A-))\\
	&\leq |V_{A,\xi}(\tau_A-) - V_{B,\xi}(\tau_A-)| 
			+ (\alpha_A(\id) + \alpha_B(\id)) (V_{A,\xi}(\tau_A-) \wedge V_{B,\xi}(\tau_A-))\\
		&\quad\ + \|\alpha_{A,B}'\|_\infty 
			\bigg(
				|y_A(\tau_A-) - \id| + |y_B(\tau_A-) - \id| + |U_A(\tau_A-)| + |U_B(\tau_A-)|
			\bigg) \\
			& \hphantom{\quad\ + \|\alpha'\|_\infty }
			\times (V_{A,\xi}(\tau_A-) \wedge V_{B,\xi}(\tau_A-))\\
	&\leq \bar{g}_{A,B}(\tau_A- ).
\end{align*}

The final term in $G_{A,B}$ is decreasing in time, because $\Vert V_{i,\xi}(\cdot,t)\Vert_1$ with $i=1,2$ is decreasing and the sets
$(\mA_i^c\cap \mB_{A,B}^c)(t)= (\mA_i\cup \mB_{A,B})^c$, with $i = A,B$, are shrinking in time,
and thus the respective indicator functions are decreasing in time.

Hence we have that $G_{A,B}(\tau) \leq G_{A,B}(\tau-)$ for all breaking times $\tau$.

\vspace{2mm}
We now wish to obtain our estimate backwards in time. We consider an arbitrary time $t$, and construct different estimates depending on what set $\xi$ is in at time $t$. As we know that $G_{A,B}(\xi,t)$ decreases over breaking times, we can employ a strategy of constructing an estimate backwards to the most recent breaking time $\tau(\xi)$, or zero if no breaking occurs in the past. Assuming we hit another breaking time, $\xi$ may enter a different set, and we can then employ our estimate for that set.

To make our strategy clearer we consider the first case, that is $\xi \in \mA_{A,B}(t)$.
In this case particle $\xi$ experienced wave breaking in the past for at least one,
or neither, of the solutions.
Set $\hat\tau(\xi)$ to be the largest of the two breaking times,
or zero if neither broke. Then
\[
	g_{A,B}(t)=|V_{A,\xi}(t) - V_{B, \xi}(t)| = |V_{A,\xi}(\hat\tau) - V_{B,\xi}(\hat \tau)|=g_{A,B}(\hat \tau).
\]
If $\hat \tau(\xi)>0$,
 depending on which set $\xi$ sat in before $\tau(\xi)$, we can employ one of our previous estimates. For example, if $\xi$ was in $\mB_{A,B}(t)$ for $t < \tau(\xi)$, we can use that
\[
	g_{A,B}(\tau) = 
	|V_{A,\xi}(\tau) - V_{B,\xi}(\tau)| 
	\leq \hat{g}_{A,B}(\tau-).
\]
We can then employ the next estimate we calculate.

Consider $\xi \in \mB_{A,B}(t)$. Then, using the estimates we have obtained in Lemma \ref{lem:ODEest}, we have
\begin{align*}
	\hat{g}_{A,B}(t) 
	&\leq \hat{g}_{A,B}(0) +\|\alpha_{A,B}'\|_\infty \int_0^t (V_{A,\xi}(s) \wedge V_{B,\xi}(s))\\
	&\hphantom{\leq \hat{g}_{A,B}(0) +\|\alpha_{A,B}'\|_\infty \int_0^t}
	\times\bigg( |U_A(s) - U_B(s)| + \frac{1}{4}\|V_{A,\xi}(s) - V_{B,\xi}(s)\|_1 \bigg)\  ds\\
	&\leq \hat{g}_{A,B}(0) +\|\alpha_{A,B}'\|_\infty \int_0^t (V_{A,\xi}(s) \wedge V_{B,\xi}(s)) \bigg( |U_A(s) - U_B(s)| + \frac{1}{4} \|G_{A,B}(s)\|_1 \bigg)\, ds \\
	&\leq \hat{g}_{A,B}(0) + \int_0^t \bigg(  \hat{g}_{A,B}(s) + \frac{1}{4} \|\alpha_{A,B}'\|_\infty (V_{A,\xi}(s) \wedge V_{B,\xi}(s))\|G_{A,B}(s)\|_1 \bigg)\, ds.
\end{align*}

Finally, we consider $\xi \in \Omega_{A,B}^c(t)$. We have
\begin{align*}
	|y_A(t) - \id|(V_{A,\xi}(t) \wedge V_{B,\xi}(t)) 
	&\leq |y_A(0) - \id|(V_{A,\xi}(0) \wedge V_{B,\xi}(0))\\
	&\quad\ + \int_0^t |U_A(s)|(V_{A,\xi}(s) \wedge V_{B,\xi}(s))\, ds
\end{align*}
and
\begin{align*}
	|U_A(t)|(V_{A,\xi}(t) \wedge V_{B,\xi}(t))
	&\leq |U_A(0)|(V_{A,\xi}(0) \wedge V_{B,\xi}(0))\\
	&\quad\ + \frac{1}{4} \int_0^t \|V_{A,\xi}(s)\|_1 (V_{A,\xi}(s) \wedge V_{B,\xi}(s))\, ds.
\end{align*}
Assume without loss of generality that $\tau_A(\xi) < \tau_B(\xi)$.
First, we consider $\tau_A(\xi) < t < \tau_B(\xi)$.
Then
\begin{align*}
	\bar{g}_{A,B}(t) 
	&= |V_{A,\xi}(t) - V_{B,\xi}(t)| 
		+ \alpha_B(\id) 
		\big(
			V_{A,\xi}(t) \wedge V_{B,\xi}(t) 
		\big) \\
	&\quad\ 
		+ \|\alpha_{A,B}'\|_\infty 
		\big(
				|y_B(t) - \id| 
			+ |U_A(t)| 
			+ |U_B(t)| 
		\big) 
		\big( 
			V_{A,\xi}(t) \wedge V_{B,\xi}(t) 
		\big)\\
	&\leq \bar{g}_{A,B}(\tau_A)
		+ \|\alpha_{A,B}'\|_\infty \int_{\tau_A}^t 
		\big( 
			V_{A,\xi}(s) \wedge V_{B,\xi}(s) 
		\big)
		\\
	&\hphantom{\leq \bar{g}(\tau_A) + \|\alpha'\|_\infty \int_{\tau_A}^t}
	\quad\ \times
		 \big(|U_{B}(s)| 
		+ \frac{1}{4}\|V_{A,\xi}(s)\|_1 
		+ \frac{1}{4}\|V_{B,\xi}(s)\|_1\big)\, ds . 
\end{align*}
For the case where $t < \tau_A(\xi) < \tau_B(\xi)$, we find
\begin{align*}
	\bar{g}_{A,B}(t)
	&= |V_{A,\xi}(t)  - V_{B,\xi}(t) |
	+(\alpha_A(\id) + \alpha_B(\id)) \big(V_{A,\xi}(t) \wedge V_{B,\xi}(t)\big)\\
	&\quad\ + \|\alpha_{A,B}'\|_\infty \big(V_{A,\xi}(t) \wedge V_{B,\xi}(t)\big) \big( |y_A(t) - \id| + |y_B(t) - \id| \\
	&\hphantom{\quad\ + \|\alpha'\|_\infty \big(V_{A,\xi}(t) \wedge V_{B,\xi}(t)\big) \big(}
	+ 2|U_A(t)| + 2|U_B(t)| \big)\\
	&\leq \bar{g}_{A,B}(0) 
		+ \|\alpha_{A,B}'\|_\infty \int_{0}^t \big( V_{A,\xi}(s) \wedge V_{B,\xi}(s) \big)\\
		&\hphantom{\bar{f}(0) + \|\alpha'\|_\infty \int_{0}^t}\quad\
		\times \big(|U_{A}(s)| + |U_B(s)| +  \frac{1}{2}\|V_{A,\xi}(s)\|_1 + \frac{1}{2}\|V_{B,\xi}(s)\|_1\big)\, ds .
\end{align*}
The case where one breaks and the other does not can be analysed in a similar manner.
In the end, we see that for any $t$ such that the final wave breaking time has not occurred, we have
\begin{align*}
	\bar{g}_{A,B}(t) 
	&\leq \bar{g}_{A,B}(0) 
		+ \|\alpha_{A,B}'\|_\infty \int_0^t (V_{A,\xi}(s) \wedge V_{B,\xi}(s)) \bigg( |U_A(s)| + |U_B(s)|\\
	&\hphantom{\leq \hat{f}(0) + \|\alpha'\|_\infty \int_0^t \bigg( } 
		+ \frac{1}{4} \|V_{A,\xi}(s)\|_1 + \frac{1}{4}\|V_{B,\xi}(s)\|_1 \bigg)(\bbo_{\mA_A^c(s)} + \bbo_{\mA_B^c(s)})\, ds\\
		& \leq \bar g_{A,B}(0)+ \int_0^t \bar g_{A,B}(s) ds\\
		& \quad +\frac14 \|\alpha_{A,B}'\|_\infty \int_0^t (V_{A,\xi}(s) \wedge V_{B,\xi}(s))\bigg(\|V_{A,\xi}(s)\|_1 + \|V_{B,\xi}(s)\|_1 \bigg)\\
		& \qquad\qquad\qquad \qquad\qquad\qquad \qquad \qquad \quad \times(\bbo_{\mA_A^c(s)} + \bbo_{\mA_B^c(s)})\bbo_{\mB_{A,B}^c(s)}\, ds\\
\end{align*}

As pointed out earlier, the final term in \eqref{eqn:G} is decreasing with respect to time.

Combining all these estimates together, we have
\begin{equation}\label{prop:ineq:FEst}
	\begin{split}
		G_{A,B}(\xi, t) 
		&\leq G_{A,B}(\xi, 0) \\
		&\quad\ 
		+ \int_0^t 
		\bigg(
			G_{A,B}(\xi, s)
				+	\frac{1}{4} \|\alpha_{A,B}'\|_\infty
				(V_{A,\xi}(\xi, s) \wedge V_{B,\xi}(\xi, s))
				\|G_{A,B}(s)\|_1 
				\bbo_{\mB_{A,B}(s)}(\xi)
		\bigg)\, ds.
	\end{split}
\end{equation}
Taking the $L^1$ norm with respect to $\xi$ of \eqref{prop:ineq:FEst}, we have
\begin{equation}
	\|G_{A,B}(t)\|_1 
	\leq \|G_{A,B}(0)\|_1 
	+ \int_0^t 
	\left(
			\|G_{A,B}(s)\|_1 + \frac14 M_{A,B} \|\alpha_{A,B}'\|_\infty \|G_{A,B}(s)\|_1
	\right)\, ds.
\end{equation}
Taking the $L^2$ norm with respect to $\xi$ of \eqref{prop:ineq:FEst}, we have
\begin{equation}
	\|G_{A,B}(t)\|_2 
	\leq \|G_{A,B}(0)\|_2 
	+ \int_0^t 
	\left(
		\|G_{A,B}(s)\|_2 + \frac14 \sqrt{M_{A,B}} \|\alpha_{A,B}'\|_\infty \|G_{A,B}(s)\|_1
	\right)\, ds,
\end{equation}
where we have used Minkowski's inequality, and that, as $|V_{A,\xi}(\xi, t)|\leq 1$ by assumption,
\[
	\int_\R (V_{A,\xi} \wedge V_{B,\xi})^2(\xi, t)\, d\xi 
	\leq \int_\R (V_{A,\xi} \wedge V_{B,\xi})(\xi, t)\, d\xi \leq M_{A,B}.
\]
\end{proof}

We then define our norm $ D:\mF \times \mF \to \R $ by
\begin{equation}\label{eqn:D}
\begin{split}
	D(X_A^{\alpha_A}, X_B^{\alpha_B}) 
	&= \|y_A - y_B\|_\infty
	+ \|U_A - U_B\|_\infty\\
	&\quad\ + \|y_{A,\xi} - y_{B,\xi}\|_2
	+ \|U_{A,\xi} - U_{B,\xi}\|_2\\
	&\quad\ + \|H_A - H_B\|_\infty 
	+ \frac{1}{4} \|G_{A,B}\|_1 
	+ \frac{1}{2} \|G_{A,B}\|_2\\
	&\quad\ + \|\alpha_A - \alpha_B\|_\infty.
\end{split}
\end{equation}

\begin{note}
	Note that $G_{A,B}$, and hence $D$, does not satisfy the triangle inequality.
	$D$, however, satisfies the other properties in the definition of a metric on the space of
	Lagrangian coordinates. Thus, it is a semi-metric.

	As we will see in Section~\ref{sec:ToAMet}, the triangle inequality is not necessary
	for our final metric construction. This is due to Lemma~\ref{lem:genMetric}.
\end{note}

\begin{lem}\label{lem:lipD}
	Let $ X_A^{\alpha_A}$ and $ X_B^{\alpha_B} $ in $ \mF $ be $\alpha$-dissipative solutions
	with initial data $X_{0,A}^{\alpha_A} \in \mF_0$ and $ X_{0,B}^{\alpha_B} \in \mF $, respectively.
	Then
	\[
		D(X_A^{\alpha_A}(t), X_B^{\alpha_B}(t)) 
		\leq e^{C_{A,B} t} D(X_{0,A}^{\alpha_A}, X_{0,B}^{\alpha_B}),
	\] 
	with
 \begin{equation}\label{const:CM}
	C_{A,B} = 2 +\frac14 \|\alpha_{A,B}'\|_\infty (M_{A,B}+2\sqrt{M_{A,B}})
\end{equation}
and $M_{A,B}$ given by \eqref{def:MAB}.
\end{lem}

\begin{proof}
We have, combining \eqref{eqn:diffVleqG} with Corollary \ref{cor:lagEsts}, 
\begin{subequations}
\begin{align*}
	\|U_A(t)-U_B(t)\|_\infty &\leq \|U_A(0) - U_B(0)\|_\infty + \frac{1}{4} \int_0^t \|G_{A,B}(s)\|_1\, ds,\\
	\|U_{A,\xi}(t)-U_{B,\xi}(t)\|_2 &\leq \|U_{A,\xi}(0) - U_{B,\xi}(0)\|_2 + \hlf \int_0^t \|G_{A,B}(s)\|_2\, ds.
\end{align*}
\end{subequations}

Combining these inequalities with our estimates from Corollary~\ref{cor:lagEsts} and Proposition~\ref{prop:G}, we have
\begin{multline}
	D(X_A^{\alpha_A}(t), X_B^{\alpha_B}(t))
	\leq D(X_{0,A}^{\alpha_A},X_{0,B}^{\alpha_B})\\
	+ \big(2 + \frac14\|\alpha_{A,B}'\|_\infty (M_{A,B}+2\sqrt{M_{A,B}}) \big) 
	\int_0^t D(X_A^{\alpha_A}(s),X_B^{\alpha_B}(s))\, ds.
\end{multline}

The result then follows from Gr\"{o}nwall's inequality.
\end{proof}

One final result we will make use of in the next section is as follows.
\begin{lem}\label{lem:JPiXAXB}
Let $X_A^{\alpha_A}$ and $X_B^{\alpha_B}$ be two $\alpha$-dissipative solutions with initial data $X_{0,A}^{\alpha_A}$
and $X_{0,B}^{\alpha_B}$ in $\mF_0$.
Given $t\geq 0$, let $f\in \mG$ such that $\Pi(X_A^{\alpha_A}(t))=X_A^{\alpha_A}(t)\circ f$ and $h\in \mG$.
Then,
\begin{align}\label{lem:JPiXaXb-Goal}
	D(\Pi(X_A^{\alpha_A}(t)),X_B^{\alpha_B}(t) \circ h)
	&= D(X_A^{\alpha_A}(t)\circ f,X_B^{\alpha_B}(t) \circ h) \\ \nonumber
	&\leq e^{(2 \bar{M}_{A,B} + \frac{1}{4})t} D(X_A^{\alpha_A}(t), X_B^{\alpha_B}(t)\circ w),
\end{align}
where  $w = h \circ f^{-1} \in \mG$ and $\bar M_{A,B}=M_{A,B}\vee 1$.
\end{lem}

\begin{proof}
To begin with note that while $h$ can be any function in $\mG$, the function $f$ is unique and depends on the chosen time $t$. In particular one has, see e.g.  \cite{grunert2021lipschitz}, that 
\begin{equation}\label{est:f1}
0\leq f_\xi(\xi)\leq e^{\frac12 t} \quad \text{ for a.e. }\xi\in \R. 
\end{equation}
Furthermore, the group property implies, that $f^{-1}(\xi)=(y_A+H_A)(\xi,t)$, and hence  \eqref{eqn:LagSys}, \eqref{ineq:ODEestU}, and $X_{0,A}^{\alpha_A}$ in $\mF_0$ yield
\begin{align}\nonumber
\vert f^{-1}(\xi)-\xi\vert &=\vert (y_A+H_A)(\xi,t)-(y_A+H_A)(\xi,0)\vert\\ \nonumber
&  \leq \int_0^t \vert U_A(\xi,s) \vert ds\\ \nonumber
&  \leq \vert U_A(\xi,t)\vert t+ \frac14 \Vert V_A(0)\Vert_\infty t^2\\ \label{eqn:finvODE}
& \leq \vert U_A(\xi,t)\vert t+\frac14 M_{A,B} t^2
\end{align}
for all $\xi\in \R$.

Keeping these estimates in mind, we drop the $t$ in $X_A^{\alpha_A}(t)$ and $X_B^{\alpha_B}(t)$ for ease in readability. 

It is immediate that
\begin{subequations}\label{proof:eqn:JPiXAXB1}
	\begin{align}
		\|y_A \circ f - y_B \circ h\|_{\infty}
		&=
		\|y_A - y_B \circ w\|_{\infty},
		\\
		\|U_A \circ f - U_B \circ h\|_{\infty}
		&=
		\|U_A - U_B \circ w\|_{\infty},
		\\
		\|H_A \circ f - H_B \circ h\|_{\infty}
		&=
		\|H_A - H_B \circ w\|_{\infty}.
	\end{align}
\end{subequations}

Note that, for any function $F:\R \to \R$ differentiable at $\xi \in \R$,
\begin{equation*}
	(F \circ h)_\xi \circ f^{-1} (\xi)
	= ( F \circ w )_\xi(\xi) f_\xi \circ f^{-1} (\xi).
\end{equation*}
Thus, after using the substitution $\eta = f(\xi)$ and \eqref{est:f1},
\begin{equation}\label{proof:eqn:JPiXAXB2}
	\begin{split}
		\|(y_A \circ f)_\xi - (y_B \circ h)_\xi\|_{2}^2 & = \int_{\R} \vert (y_A\circ f)_\xi-(y_B\circ h)_\xi\vert ^2\circ f^{-1} (f^{-1})_\xi (\eta) d\eta \\
		&=
		\int_{\R}^{}
		|y_{A,\xi}(\eta) - ( y_B \circ w )_\xi(\eta)|^2 f_\xi \circ f^{-1}(\eta)
		\ d\eta
		\\
		&\leq 
		\|f_\xi\|_\infty \|y_{A,\xi} - (y_B \circ w)_\xi\|_2^2\\
		& \leq e^{\frac12 t}\|y_{A,\xi} - (y_B \circ w)_\xi\|_2^2.
	\end{split}
\end{equation}
And similarly, one finds
\begin{equation}\label{proof:eqn:JPiXAXB3}
	\|(U_A \circ f)_\xi - (U_B \circ h)_\xi\|_{2}^2
	\leq
	e^{\frac12 t} \|U_{A,\xi} - (U_B \circ w)_\xi\|_2^2.
\end{equation}

We wish to show that
\begin{equation}\label{ineq:toProveG}
	\begin{split}
		G(X_A^{\alpha_A} \circ f, X_B^{\alpha_B} \circ h) \circ f^{-1}(\xi)\leq A(t) G(X_A^{\alpha_A}, X_B^{\alpha_B}\circ w)(\xi) f_\xi \circ f^{-1}(\xi),
	\end{split}
\end{equation}
for some positive function $A(t)$.

For the characteristic functions, we have
\begin{equation}\label{eqn:1Rel}
\begin{array}{rcl@{ }}
	\bbo_{\mA_{A,B}^{f,h}}\circ f^{-1} = \bbo_{\mA_{A,B}^{\id, w}},
	& \bbo_{\mB_{A,B}^{f,h}}\circ f^{-1} = \bbo_{\mB_{A,B}^{\id, w}},
\end{array}
\end{equation}
and 
\begin{equation}\label{eqn:1Rel2}
	\bbo_{\Omega_{A,B}^{f,h,c}} \circ f^{-1}
	= \bbo_{\Omega_{A,B}^{\id , w, c}},
\end{equation}
which follow from \eqref{eqn:SetRelAfh}, \eqref{eqn:SetRelBfh}, and \eqref{eqn:SetRelOmfh}.

For $g$ and $\hat g$, given by \eqref{eqn:g} and \eqref{eqn:ghat},
\begin{align}\label{eqn:gThm1}
	g(X_A^{\alpha_A}\circ f, X_B^{\alpha_B}\circ h) \circ f^{-1} &= |V_{A,\xi} - (V_B \circ w)_\xi| f_\xi \circ f^{-1}\\ \nonumber
	& = g(X_A^{\alpha_A}, X_B^{\alpha_B}\circ w)f_\xi\circ f^{-1},
\end{align}
\begin{equation}\label{eqn:gThm2}
\begin{split}
	\hat{g}(X_A^{\alpha_A}\circ f, X_B^{\alpha_B}\circ h) \circ f^{-1} 
	= \bigg[&
		|V_{A,\xi} - (V_B \circ w)_\xi|
		+ \|\alpha_A - \alpha_B\|_\infty \big(V_{A,\xi} \wedge (V_B\circ w)_\xi \big)\\
					&+ \|\alpha_{A,B}'\|_\infty \big(V_{A,\xi} \wedge (V_B\circ w)_\xi \big)\\
					&\quad\ \times
					\Big(
						|y_{A} - y_B\circ w| + |U_{A} -U_{B}\circ w|
					\Big)
				\bigg] f_\xi \circ f^{-1}\\
				& =\hat g(X_A^{\alpha_A}, X_B^{\alpha_B}\circ w)f_\xi\circ f^{-1}.
\end{split}
\end{equation}

For $\bar g$ given by  \eqref{eqn:gbar}, 
\begin{equation}\label{eqn:gThm3}
\begin{split}
	\bar{g}(X_A^{\alpha_A}\circ f, X_B^{\alpha_B}\circ h) \circ f^{-1} 
	= \bigg[&
	|V_{A,\xi} - (V_B \circ w)_\xi|\\
	&+ \big(V_{A,\xi} \wedge (V_B\circ w)_\xi \big) ((\alpha_A \circ f^{-1})\bbo_{\mA_A^c} + (\alpha_B\circ f^{-1})\bbo_{\mA_B^{w,c}})\\
	&+ \|\alpha_{A,B}'\|_\infty \big(V_{A,\xi} \wedge (V_B\circ w)_\xi \big)\\
	&\quad\ \times
		\Big(
		|y_A - f^{-1}|\bbo_{\mA_A^c} + |y_B\circ w - f^{-1}|\bbo_{\mA_B^{w,c}}\\
		&\quad\ \quad\
		+ (|U_A| + |U_B\circ w|)(\bbo_{\mA_A^c} + \bbo_{\mA_B^{w,c}}) 
		\Big)
	\bigg]f_\xi \circ f^{-1}\\
		\leq \bigg[&
			|V_{A,\xi} - (V_B \circ w)_\xi|\\
							 &+ \big(V_{A,\xi} \wedge (V_B\circ w)_\xi \big) \\
							 &\quad\ \times \big((\|\alpha_{A,B}'\|_\infty(|U_A|t + \frac{1}{4}M_{A,B}t^2) + \alpha_A)\bbo_{\mA_A^c} \\
							 &\quad\ \quad\ + (\|\alpha_{A,B}'\|_\infty(|U_A|t + \frac{1}{4}M_{A,B}t^2) + \alpha_B)\bbo_{\mA_B^{w,c}}\big)\\
							 &+ \|\alpha_{A,B}'\|_\infty \big(V_{A,\xi} \wedge (V_B\circ w)_\xi \big)\\
							 &\quad\ \times
							 \Big(
								 (|y_A - \id| + |U_A|t + \frac{1}{4}M_{A,B}t^2 )\bbo_{\mA_A^c} \\
							 &\quad\ \quad\ + (|y_B\circ w - \id| + |U_A| t + \frac{1}{4}M_{A,B}t^2)\bbo_{\mA_B^{w,c}}\\
							 &\quad\ \quad\
							 + (|U_A| + |U_B\circ w|)(\bbo_{\mA^c} + \bbo_{\mA_B^{w,c}}) 
						 \Big)
					 \bigg]f_\xi \circ f^{-1},\\
					&=\bar  g(X_A^{\alpha_A}, X_B^{\alpha_B}\circ w)f_\xi\circ f^{-1}\\
					& \quad\quad  +  2\|\alpha_{A,B}'\|_\infty \big(V_{A,\xi} \wedge (V_B\circ w)_\xi \big)\\
					& \quad\qquad \qquad  \quad \times (|U_A|t + \frac{1}{4}M_{A,B}t^2 )(\bbo_{\mA_A^c}+\bbo_{\mA_B^{w,c}})f_\xi\circ f^{-1} ,
		\end{split}
\end{equation}
where we used \eqref{eqn:finvODE}.
Finally, for the last term in $ G $, we apply the same strategy. We get
\begin{equation}\label{eqn:gThm4}
	\begin{split}
		&\bigg[
			\big(
				(V_A \circ f)_\xi \wedge (V_B\circ h)_\xi
			\big)(
			\|(V_A\circ f)_\xi\|_1 + \|(V_B \circ h)_\xi\|_1 + 1
			)(
			\bbo_{\mA_A^{f,c}} + \bbo_{\mA_B^{h,c}}
			) \bbo_{\mathcal{B}_{A,B}^{f,h}} 
		\bigg]\circ f^{-1}\\
		&= \big(
			V_{A,\xi} \wedge (V_B\circ w)_\xi
		\big)(
		\|V_{A,\xi}\|_1 + \|(V_B \circ w)_\xi\|_1 + 1
		)(
		\bbo_{\mA_A^{c}} + \bbo_{\mA_B^{w,c}}
		) \bbo_{\mathcal{B}_{A,B}^{\id,w}} 
		f_\xi \circ f^{-1},
	\end{split}
\end{equation}
where we have used substitution to deal with the $L^1(\R)$ terms present inside this term.

Thus, combining~\eqref{eqn:1Rel}, \eqref{eqn:1Rel2}, \eqref{eqn:gThm1}, \eqref{eqn:gThm2},
\eqref{eqn:gThm3}, and~\eqref{eqn:gThm4}, we find
\begin{equation*}
	G(X_A^{\alpha_A}\circ f, X_B^{\alpha_B} \circ h) \circ f^{-1} 
	\leq (1 + 2t +2M_{A,B}t^2)G(X_A^{\alpha_A}, X_B^{\alpha_B}\circ w,) f_\xi \circ f^{-1},
\end{equation*}
exactly as desired in~\eqref{ineq:toProveG}.
Taking the $L^1(\R)$ and $L^2(\R)$ norms, and with the substitution $\eta = f(\xi)$, we have
\begin{equation}\label{proof:eqn:JPiXAXB4}
	\|G(X_A^{\alpha_A}\circ f, X_B^{\alpha_B} \circ h)\|_1	\\
	\leq (1 + 2t + 2M_{A,B}t^2) \|G(X_A^{\alpha_A}, X_B^{\alpha_B}\circ w)\|_1
\end{equation}
and
\begin{equation}\label{proof:eqn:JPiXAXB5}
	\|G(X_A^{\alpha_A}\circ f, X_B^{\alpha_B} \circ h)\|_2  \\
	\leq (1 + 2t + 2M_{A,B}t^2) e^{\frac14 t}\|G(X_A^{\alpha_A}, X_B^{\alpha_B}\circ w)\|_2.
\end{equation}

Combining~\eqref{proof:eqn:JPiXAXB1}, \eqref{proof:eqn:JPiXAXB2}, \eqref{proof:eqn:JPiXAXB3},
\eqref{proof:eqn:JPiXAXB4}, and \eqref{proof:eqn:JPiXAXB5}, we have
\begin{align*}
	D(X_A^{\alpha_A}\circ f, &X_B^{\alpha_B} \circ h )\\
	&\leq (1+2t + 2M_{A,B}t^2)e^{\frac14 t}D(X_A^{\alpha_A}, X_B^{\alpha_B}\circ w) \\
	&\leq e^{2 \bar{M}_{A,B} t}e^{\frac14 t}D(X_A^{\alpha_A}, X_B^{\alpha_B}\circ w).
\end{align*}

\end{proof}

\section{Towards a metric}\label{sec:ToAMet}
We have two issues we strive to resolve in this section.
First, the mapping constructed in the previous section is not a metric,
but it is a semi-metric.
Second, Lagrangian coordinates that represent the same Eulerian coordinates,
i.e. lie in the same equivalence class, do not in general have a distance of zero. 
In other words, this is a semi-metric over the whole space of Lagrangian coordinates,
but not over the space of equivalence classes. In resolving the second issue, we resolve the first.

We begin with a helpful observation from the proof of Lemma~\ref{lem:ODEest} and \eqref{eqn:diffVleqG}.
\begin{prop}\label{prop:VinfVxi1}
	Let $X^{\alpha_A}_A$ and $X^{\alpha_B}_B$ be in $\mF$. Then
	\[
		\|V_A - V_B\|_\infty \leq \|V_{A,\xi} - V_{B,\xi}\|_1\leq \|G_{A,B}\|_1.
	\]
\end{prop}

Define $J:\mF \times \mF \to \R$ by 
\begin{equation}
	J(X_A^{\alpha_A}, X_B^{\alpha_B})
	= \inf_{f,g\in \mG} 
	\left(
		D(X_A^{\alpha_A}, X_B^{\alpha_B} \circ f)
		+ D(X_A^{\alpha_A} \circ g, X_B^{\alpha_B}) 
	\right).
\end{equation}

We begin by noting that $J$ is zero when measuring the distance between members of the same equivalence class.
Indeed, suppose that $X_A^{\alpha_A}$ and $X_B^{\alpha_B}$ in $\mF$ share the same equivalence class. 
That is, there exist $f_A$ and $f_B$ in $\mathcal{G}$ such that 
\begin{equation*}
	X_A^{\alpha_A} \circ f_A=X_B^{\alpha_B} 
	\quad \text{ and }\quad 
	X_B^{\alpha_B}\circ f_B=X_A^{\alpha_A}.
\end{equation*}
Then
\begin{equation*}
	D(X_A^{\alpha_A}, X_B^{ \alpha _B} \circ f_B) 
		+ D(X_A^{\alpha_A} \circ f_A, X_B^{ \alpha _B})
	= D(X_A^{ \alpha_A }, X_A^{ \alpha_A}) + D(X_B^{\alpha_B}, X_B^{\alpha_B})
	= 0,
\end{equation*}
and hence the infimum will be zero.

We will make use of a slight modification of a result that has already been established in \cite[Lemma 3.2]{MR3007728}.
\begin{lem}
	Let $X_A^{\alpha_A}$ and $X_B^{\alpha_B}$ be in $\mF_0$. 
Then, for any relabelling function $f \in \mathcal{G}$,
\[
	\|X_A^{\alpha_A} - X_B^{\alpha_B}\| \leq 5 \|X_A^{\alpha_A} \circ f - X_B^{\alpha_B}\|,
\]
where the norm $\|\cdot\|$ is given by,
\[
	\|X^\alpha\| = \|y-\id\|_\infty + \|U\|_\infty + \|H\|_\infty +\frac14  \|V\|_\infty+\|\alpha\|_\infty
	\quad \text{ for any } X^\alpha \in \mF.
\]
\end{lem}

Hence, we have
\begin{equation}\label{eqn:XAXB1}
	2\|X_A ^{\alpha_A}- X_B^{\alpha_B}\| \leq 5 \|X_A^{\alpha_A} \circ f - X_B^{\alpha_B}\| + 5 \|X_A^{\alpha_A} - X_B^{\alpha_B} \circ g\|.
\end{equation}
Using Proposition \ref{prop:VinfVxi1}, we have
\begin{equation}\label{eqn:XAcFXB}
\begin{split}
	\|X_A^{\alpha_A} \circ f - X_B^{\alpha_B}\| 
	&= \|y_A \circ f - y_B\|_\infty + \|U_A \circ f - U_B\|_\infty\\
	&\quad\ +\|H_A \circ f - H_B\|_\infty +\frac14 \|V_A \circ f - V_B\|_\infty+\|\alpha_A-\alpha_B\|_\infty \\
	&\leq \|y_A \circ f - y_B\|_\infty + \|U_A \circ f - U_B\|_\infty\\
	&\quad\ +\|H_A \circ f - H_B\|_\infty + \frac14 \|G(X_A^{\alpha_A} \circ f, X_B^{\alpha_B})\|_1+\|\alpha_A-\alpha_B\|_\infty\\
	&\leq D(X_A^{\alpha_A} \circ f, X_B^{\alpha_B}).
\end{split}
\end{equation}
Thus, substituting this inequality into \eqref{eqn:XAXB1}, we see
\begin{equation}\label{ineq:XAXBinf}
	\|X_A^{\alpha_A} - X_B^{\alpha_B}\| 
	\leq \frac{5}{2}
	\big( 
		D(X_A^{\alpha_A} \circ f, X_B^{\alpha_B}) 
		+ D(X_A^{\alpha_A} , X_B^{\alpha_B} \circ g) 
	\big)
\end{equation}
and after taking the infimum over all $f,g \in \mG$, we have the following.
\begin{cor}\label{cor:ineq:XnrmJ}
	Let $X_A^{\alpha_A}$ and $X_B^{\alpha_B}$ be in $\mF_0$. Then
\[
	\|X_A^{\alpha_A} - X_B^{\alpha_B}\|
	\leq \frac{5}{2}J(X_A^{\alpha_A}, X_B^{\alpha_B}).
\]
Thus the restriction of $J$ to $\mF_0 \times \mF_0$ is a semi-metric.
\end{cor}

Using this semi-metric we are able to construct a metric on the more restricted set $\mF_M^{L}$,
given by~\eqref{eqn:FM}. For $M$, $L>0$,
introduce $\hat{d}:\mF_M^{L} \times \mF_M^{L} \to \R$, defined by
\begin{equation}\label{eqn:lagMet}
	\hat{d}(X_A^{\alpha_A}, X_B^{\alpha_B}) 
	\coloneqq 
	\inf_{ \hat{\mF}\left(X_A^{\alpha_A}, X_B^{\alpha_B}\right) } 
	\sum_{n=1}^N J(X_n^{\alpha_n}, X_{n-1}^{\alpha_{n-1}}),
\end{equation}
where $ \hat{\mF}\left(X_A^{\alpha_A}, X_B^{\alpha_B}\right) $ is the set of finite sequences 
$\{X_n^{\alpha_n}\}_{n=0}^N$ of arbitrary length in $\mF_{0,M}^{L}$,
satisfying $X_0^{\alpha_0} = \Pi X_A^{\alpha_A}$
and $X_N^{\alpha_N} = \Pi X_B^{\alpha_B}$.

\begin{note}
	Let $X_A^{\alpha_A}$, $X_B^{\alpha_B} \in \mF_M^L$. Then, directly from the definition we have
\[
	\hat{d}(X_A^{\alpha_A}, X_B^{\alpha_B}) 
	= \hat{d}(\Pi X_A^{\alpha_A}, \Pi X_B^{\alpha_B}).
\]
\end{note}

\begin{note}
	$\hat{d}$ inherits from $J$ that if both $X_A^{\alpha_A}$ and $X_B^{\alpha_B}$ are in the same equivalence
	class, then $\hat{d}(X_A^{\alpha_A}, X_B^{\alpha_B})$ is zero. Indeed, consider the
	finite sequence $X_0^{\alpha_0} = \Pi X_A^{\alpha_A}$
	and $X_1^{\alpha_1} = \Pi X_B^{\alpha_B} = \Pi X_A^{\alpha_A}$.
\end{note}

It remains to ensure that $\hat{d}$ satisfies the identity of indiscernibles. That is we need to prove the implication 
\begin{equation}\label{eqn:J0Equiv}
	\hat{d}(X_A^{\alpha_A}, X_B^{\alpha_B}) = 0 
	\implies X_A^{\alpha_A} \sim X_B^{\alpha_B} ,
\end{equation}
meaning if the distance between the two elements is zero,
then both Lagrangian coordinates lie in the same equivalence class.

Using Corollary~\ref{cor:ineq:XnrmJ} and Lemma~\ref{lem:genMetric}, with $F = J$,we get the following result, which confirms~\eqref{eqn:J0Equiv}.
\begin{cor}\label{cor:disMet}
The function $ \hat{d}:\mF_M^{L} \times \mF_M^{L} \to \R $ defined by \eqref{eqn:lagMet} is a metric. Furthermore, for any $X_A^{\alpha_A}, X_B^{\alpha_B} \in \mF_{0,M}^{L}$ it satisfies
\[
	\frac{2}{5}
		\|X_A^{\alpha_A} - X_B^{\alpha_B}\| 
	\leq \hat{d}(X_A^{\alpha_A}, X_B^{\alpha_B}) 
	\leq J(X_A^{\alpha_A}, X_B^{\alpha_B}).
\]
\end{cor}

The following lemma will form the bridge that allows us to use the Lipschitz
stability estimate we have obtained for $D$ to
prove Lipschitz stability with respect to $\hat{d}$.

\begin{lem}\label{lem:JPiEst}
Let $X_A^{\alpha_A}$ and $X_B^{\alpha_B}$ be two $\alpha$-dissipative solutions with initial
data $X_{0,A}^{\alpha_A}$ and $X_{0,B}^{\alpha_B}$ in $\mF_0$, respectively.
Then
\[
	J(\Pi X_A^{\alpha_A}(t), \Pi X_B^{\alpha_B}(t))
	\leq e^{ (4\bar{M}_{A,B} + \hlf) t} J(X_A^{\alpha_A}(t), X_B^{\alpha_B}(t)),
\]
where $\bar{M}_{A,B}=M_{A,B}\vee1$.
\end{lem} 

\begin{proof}
To ease digestion, we drop $\alpha$ and $t$ in the notation for this proof.
Furthermore, we set $C \coloneqq 2 \bar{M}_{A,B} + \frac{1}{4}$ and for $i=A$, $B$, let $f_i\in \mG$ such that $\Pi X_i^{\alpha_i}=X_i^{\alpha_i}\circ f_i$.

From Lemma~\ref{lem:JPiXAXB}, we have,
\begin{equation}\label{eqn:Jprf1}
\begin{split}
	J(\Pi X_A, \Pi X_B) &
	= \inf_{f_1,f_2}(D(X_A\circ f_A \circ f_1,\Pi X_B) + D(X_A\circ f_A, (\Pi X_B) \circ f_2))\\
	&\leq \inf_{f_1,f_2}(D(X_A\circ f_1, \Pi X_B) + e^{C t} D(X_A, (\Pi X_B) \circ f_2 \circ f_A^{-1} ))\\
	&\leq e^{C t} \inf_{f_1,f_2}(D(X_A\circ f_1, \Pi X_B) + D(X_A, (\Pi X_B) \circ f_2 ))\\
	&= e^{C t} J(X_A, \Pi X_B),
\end{split}
\end{equation}
where we are using that $f_A \circ f_1$ lies in $\mG$ for any $f_1\in \mG$ and that any element $f\in\mG$ can be written as $f=f_A\circ g$ for some $g\in \mG$, which implies that $g=f_A^{-1}\circ f$.

We can then do the same again, but now we swap the roles of $X_B$ and $X_A$,
\begin{equation}\label{eqn:Jprf2}
\begin{split}
	J(X_A, \Pi X_B) &
	= \inf_{f_1,f_2}(D(X_A\circ f_1, X_B \circ f_B) + D(X_A,  X_B \circ f_B \circ f_2))\\
	&\leq \inf_{f_1,f_2}(e^{C t}D(X_A \circ f_1 \circ f_B^{-1}, X_B) + D(X_A, X_B \circ f_2 ))\\
	&\leq e^{C t} \inf_{f_1,f_2}(D(X_A\circ f_1, X_B) + D(X_A, X_B \circ f_2 ))\\
	&= e^{C t} J(X_A, X_B).
\end{split}
\end{equation}

Substituting \eqref{eqn:Jprf2} into \eqref{eqn:Jprf1}, we obtain the required result.
\end{proof}

Thus we can now show our Lipschitz stability result.
\begin{thm}\label{thm:LagStab}
Let $X_A^{\alpha_A}$ and $X_B^{\alpha_B}$ be two $\alpha$-dissipative solutions with initial
data $X_{0,A}^{\alpha_A}$ and $X_{0,B}^{\alpha_B}$ in $\mF_{0,M}^{L}$, respectively.
Then
\begin{equation}
	\hat{d}(X_A^{\alpha_A}(t), X_B^{\alpha_B}(t))
	\leq e^{R_M^{L} t}\hat{d}(X_{0,A}^{\alpha_A}, X_{0,B}^{\alpha_B}),
\end{equation}
where
\begin{equation}\label{const:RM}
	R_M^{L} \coloneqq 4\bar{M} +\frac52 + \frac{1}{4}L(M +2 \sqrt{M}),
\end{equation}
with $\bar{M} = M \vee 1$. 
\end{thm}

\begin{proof}
Let $\eps > 0$. Consider a finite sequence 
$ \{ X_{0,n}^{\alpha_n} \}_{n=0}^N \in \hat{\mF}(X_{0,A}^{\alpha_A}, X_{0,B}^{\alpha_B}) $ 
and a sequence of relabelling functions $ \{f_n\}_{n=0}^{N-1}, \{g_n\}_{n=1}^{N}$ in $\mG$
such that
\begin{equation*}
	\sum_{n=1}^N 
	\big(
		D(X_{0,n}^{\alpha_n}, X_{0,n-1}^{\alpha_{n-1}} \circ f_{n-1}) 
		+ D(X_{0,n}^{\alpha_n} \circ g_n, X_{0,n-1}^{\alpha_{n-1}})
	\big)
	< \hat{d}(X_{0,A}^{\alpha_A}, X_{0,B}^{\alpha_B}) + \eps.
\end{equation*}
Set $X_n^{\alpha_n}(t) = S_t X_{0,n}^{\alpha_n}$. Then, by Lemma~\ref{lem:Stof},
$X_A^{\alpha_A}(t) = S_t X_{0,0}^{\alpha_0}$, and
$X_B^{\alpha_B}(t) = S_t X_{0,N}^{\alpha_N}$.
Furthermore, $X_n^{\alpha_n}(t)\in \mF_M^L$ for all $t\geq 0$ and all $n$.
Thus, using Lemmas~\ref{lem:Stof}, \ref{lem:lipD} and~\ref{lem:JPiEst},
\begin{align*}
	\hat{d}(X_A^{\alpha_A}(t), X_B^{\alpha_B}(t)) 
	&\leq \sum_{n=1}^N J(\Pi X_n^{\alpha_n}(t), \Pi X_{n-1}^{\alpha_{n-1}}(t))\\
	&\leq e^{(4\bar M+\frac12) t} \sum_{n=1}^N J(X_n^{\alpha_n}(t), X_{n-1}^{\alpha_{n-1}}(t))\\
	&\leq e^{(4\bar M+\frac12) t} \sum_{n=1}^{N} \big( D(X_n^{\alpha_n}(t), X_{n-1}^{\alpha_{n-1}}(t)\circ f_{n-1}) \\
	&\hphantom{\leq e^{(4\bar M + \hlf) t} \sum_{n=1}^{N} \big(}
		+ D(X_n^{\alpha_n}(t) \circ g_n, X_{n-1}^{\alpha_{n-1}}(t)) \big)\\
	&\leq e^{(4\bar M+\frac12) t} \sum_{n=1}^{N} e^{(2+\frac14 L(M+2\sqrt{M}))t} \big( D(X_n^{\alpha_n}(0), X_{n-1}^{\alpha_{n-1}}(0)\circ f_{n-1}) \\
	&\qquad\quad  \qquad \hphantom{\leq e^{(4M + \hlf) t} e^{C_M t} \sum_{n=1}^{N} \big(}
		+ D(X_n^{\alpha_n}(0) \circ g_n, X_{n-1}^{\alpha_{n-1}}(0)) \big)\\
	&\leq e^{(4\bar M+\frac52+\frac14 L(M+2\sqrt{M}))t} ( \hat{d}(X_{0,A}^{\alpha_A}, X_{0,B}^{\alpha_B}) + \eps.).
\end{align*}
The final result follows, as this inequality is true for any $ \eps > 0 $.
\end{proof}

\subsection{A simplification in the case $\alpha$ is a constant}\label{subsec:ac}
In the case where $\alpha \in [0, 1] \subset \Lambda$, i.e. $\alpha$ is a constant,
the construction can be simplified.

First, define the subset of $\mF$ containing elements for which $\alpha$ is constant,
\[
	\mF_c \coloneqq \{X^{\alpha} \in \mF \mid \alpha \in [0, 1]\}.
\]

For $X^{\alpha} \in \mF_c$, we introduce the two functions
\[
	V_\xi^d(\xi, t) \coloneqq \alpha V_\xi(\xi, t) \bbo_{\mA^c(t)}(\xi),
	\quad\
	V_\xi^c(\xi, t) \coloneqq (1 - \alpha \bbo_{\mA^c(t)}(\xi)) V_\xi(\xi, t).
\]
The second function $V_\xi^c$ is in fact constant, so the time dependence can be dropped.
Note also $V_\xi(\xi, t) = V_\xi^c(\xi) + V_\xi^d(\xi, t)$.

Using this, we can introduce a simpler function $G:\mF_c\times\mF_c \to [0, +\infty)$, given by
\begin{align*}
	G_{A,B}(\xi) 
	&= G\left(X_A^{\alpha_A}, X_B^{\alpha_B}\right)(\xi)\\
	&= |V_{A,\xi}(\xi) - V_{B,\xi}(\xi)| \bbo_{\mA_{A,B}}(\xi)\\
	&\quad\ +
		\left(|V_{A,\xi}^c(\xi) - V_{B,\xi}^c(\xi)| + |V_{A,\xi}^d(\xi) - V_{B,\xi}^d(\xi)| \right)\bbo_{\mB_{A,B}}(\xi)\\
	&\quad\ +
		\left(|V_{A,\xi}^c(\xi) - V_{B,\xi}^c(\xi)| + V_{A,\xi}^d(\xi) \vee V_{B,\xi}^d(\xi) \right)\bbo_{\Omega_{A,B}^c}(\xi),
\end{align*}
for any $ X^{\alpha_A}_A$, $X^{\alpha_B}_B \in \mF^c $.
It satisfies
\[
	|V_{A,\xi}(\xi) - V_{B,\xi}(\xi)| \leq |G_{A,B}(\xi)|.
\]

We can then define a metric $D: \mF_c\times \mF_c \to \R$ by
\begin{equation}\label{eqn:DalphConst}
\begin{split}
	D(X_A^{\alpha_A}, X_B^{\alpha_B})
	&\coloneqq \|y_A - y_B\|_\infty + \|U_A - U_B\|_\infty + \|H_A - H_B\|_\infty\\
		&\quad\ + \|y_{A,\xi} - y_{B,\xi}\|_2 + \|U_{A,\xi} - U_{B, \xi}\|_2\\
		&\quad\ + \frac{1}{4}\|G_{A,B}\|_1 + \frac{1}{2}\|G_{A,B}\|_2 + |\alpha_A - \alpha_B|.
\end{split}
\end{equation}

The construction throughout Section 3 and Section 4 can be repeated, see \cite[Section 2.5]{thesis}, yielding the following result. For any two $\alpha$-dissipative solutions 
$X_A^{\alpha_A}$, $X_B^{\alpha_B}$
with initial data $ X_{0,A}^{\alpha_A}$, $ X_{0,B}^{\alpha_B} \in \mF_c \cap \mF_0$, 
\[
	\hat{d}(X_A^{\alpha_A}(t), X_B^{\alpha_B}(t)) 
	\leq e^{\frac{3}{2}t}\hat{d}(X_{0,A}^{\alpha_A}, X_{0,B}^{\alpha_B}).
\]
Note that here $L=0$ and that the exponent is independent of $M$. This is also why we can consider any initial data in  $\mF_c\cap \mF_0$ and not only in $\mF_c\cap \mF_{0,M}^L$.

\section{A return to Eulerian coordinates}
Using our metric in Lagrangian coordinates,
we shall now define our metric in Eulerian coordinates.
The problem we have to overcome is the fact that a solution to the
$\alpha$-dissipative Hunter--Saxton problem consists of a pair $ (u, \mu) $,
and the additional dummy measure $\nu$ is only necessary for the construction
of said solution. 

Before we tackle this issue, we note an immediate corollary of our previous theorem.
Define the metric $ d_\mD: \mD_M^{L} \times \mD_M^{L} \to \R $ by
\begin{equation}
	d_{\mathcal{D}}(Y_A^{\alpha_A}, Y_B^{\alpha_B})
	\coloneqq
	\hat{d}(\Lag(Y_A^{\alpha_A}), \Lag(Y_B^{\alpha_B})).
\end{equation}
We then have the following result which is an immediate consequence of Theorem \ref{thm:LagStab}.

\begin{cor}\label{cor:lipEstdD}
Let $Y_A^{\alpha_A}$, $Y_B^{\alpha_B}$ be two $\alpha$-dissipative solutions with initial data 
	 $Y_{0,A}^{\alpha_A}$ and  $Y_{0,B}^{\alpha_B}$ in  $\mD_M^{L}$, respectively. Then
	 	\[
		d_\mD(Y_A^{\alpha_A}(t), Y_B^{\alpha_B}(t)) 
		\leq e^{R_M^{L} t} d_\mD(Y_{0,A}^{\alpha_A}, Y_{0,B}^{\alpha_B}),
	\]
	with $R_M^{L}$ given by~\eqref{const:RM}.
\end{cor}

Recalling Definition~\ref{def:equiv:euler}, our construction now follows a very similar path to that of the Lagrangian metric. 
We begin by defining a function 
$\hat{J}: \mD_{0, M}^{L} \times \mD_{0, M}^{L} \to \R$, given by
\begin{equation}\label{eqn:JM}
	\hat{J}(Z_A^{\alpha_A}, Z_B^{\alpha_B})
	\coloneqq \inf_{(\nu_A, \nu_B) \in \mathcal{V}(Z_A^{\alpha_A}) \times \mathcal{V}(Z_B^{\alpha_B})}
		d_\mD(
			((Z_A, \nu_A), \alpha_A), 
			((Z_B, \nu_B), \alpha_B)
		),
\end{equation}
which no longer depends on the choice of $\nu$.
In a similar vain to $J$, this function is zero when measuring the distance between two elements of the same equivalence class in $\mD$. 

We cannot conclude that $\hat{J}$ satisfies the triangle inequality.
Using the same strategy as before,
we define the function 
$\bar{d}: \mD_{0,M}^{L} \times \mD_{0,M}^{L}\to \R$ by 
\begin{equation}
	\bar{d}\left(
		Z_A^{\alpha_A}, 
		Z_B^{\alpha_B}
	\right) 
	= \inf_{\hat{\mD}(Z_A^{\alpha_A}, Z_B^{\alpha_B})} \sum_{i=1}^N 
		\hat{J}(
		Z_i^{\alpha_i},Z_{i-1}^{\alpha_{i-1}}),
\end{equation}
where $ \hat{\mD}(Z_A^{\alpha_A}, Z_B^{\alpha_B}) $ denotes the set of 
all finite sequences of arbitrary length 
$ \big\{Z_i^{\alpha_i} \big\}_{i=0}^N $ in $\mD_{0,M}^{L}$ 
satisfying  $Z_0^{\alpha_0} =Z_A^{\alpha_A}$ and $Z_N^{\alpha_N}= Z_B^{\alpha_B}$.

From Lemma~\ref{lem:genMetric}, we can only conclude that $\bar{d}$ is a pseudo-metric, as inequality \eqref{lem:ineq:Fineq} is not satisfied. It therefore remains to prove the implication 
\begin{equation*}
\bar d(Z_A^{\alpha_A}, Z_B^{\alpha_B})=0 \implies  Z_A^{\alpha_A}= Z_B^{\alpha_B}.
\end{equation*}

We introduce now the bounded Lipschitz norm on the set of finite Radon measures $\mathcal{M}(\R)$,
\begin{equation}\label{eqn:BddLip}
	\|\mu\|_\mathcal{M} = \sup_{\phi \in \mathcal{L}} \left| \int_\R \phi(x)\, d\mu \right|,
\end{equation}
where
\[
	\mathcal{L} = \{\phi \in W^{1,\infty}(\R) \mid \|\phi\|_{1,\infty} \leq 1 \}.
\]

\begin{lem}
For $Z_A^{\alpha_A}=((u_A, \mu_A), \alpha_A )$ and $Z_B^{\alpha_B}= ((u_B, \mu_B), \alpha_B)$ in $\mD_{0}$, define the norm
\[
	\|Z_A^{\alpha_A}-Z_B^{\alpha_B}\|_{\mD_{0}} 
	\coloneqq \|u_A - u_B\|_\infty + \|\mu_A - \mu_B\|_\mathcal{M}+\|\alpha_A-\alpha_B\|_{\infty}.
\]
Then, for any $Z_A^{\alpha_A}$, $Z_B^{\alpha_B} \in \mD_{0,M}^{L},$
\begin{equation}\label{ineq:genMetric1}
	\|Z_A^{\alpha_A}-Z_B^{\alpha_B}\|_{\mD_{0}}
	\leq \left(5 + 2\bar{M} \right) 
	\bar{d}(
		Z_A^{\alpha_A}, Z_B^{\alpha_B}	 )+ \frac{\sqrt{5M}}{\sqrt{2}}\sqrt{
	\bar{d}(
		Z_A^{\alpha_A}, Z_B^{\alpha_B}
	)},
\end{equation}
where $\bar{M} = 1 \vee M$.
\end{lem}
\begin{proof}

Let $\eps > 0$. Consider a sequence 
$$\{Y_k^{\alpha_k}\}_{k=0}^N=\{((Z_k, \nu_k), \alpha_k)\}_{k=0}^N=\{((u_k,\mu_k, \nu_k),\alpha_k)\}_{k=0}^N\quad \text{ in }\mD_M^L $$
satisfying $Z_0^{\alpha_0}=Z_A^{\alpha_A}$ and $Z_N^{\alpha_N}=Z_B^{\alpha_B}$
such that
\begin{equation*}
	\sum_{k=1}^N 
		d_\mD(
			Y_k^{\alpha_k},Y_{k-1}^{\alpha_{k-1}}
		)
	\leq \bar{d}(
		Z_A^{\alpha_A}, Z_B^{\alpha_B}
	) + \eps.
\end{equation*}

Set $X_k^{\alpha_k} = \Lag(Y_k^{\alpha_k})$ for $k=0,\dots,N$. Notice that from the definition of $\Lag$, $ X_k^{\alpha_k} \in \mF_0 $.
Then, from Corollary \ref{cor:disMet}
\begin{align}\nonumber
	\|\alpha_A-\alpha_B\|_\infty\leq \|X_0^{\alpha_0} - X_N^{\alpha_N}\| 
	&\leq \frac{5}{2}\hat{d}(X_0^{\alpha_0}, X_N^{\alpha_N}) \\ \nonumber
	&= \frac{5}{2} d_\mD(Y_0^{\alpha_0}, Y_N^{\alpha_N}) \\ \nonumber
	&\leq \frac{5}{2} \sum_{k=1}^N d_\mD(Y_k^{\alpha_k}, Y_{k-1}^{\alpha_{k-1}}) \\\label{inequ:vik}
	&\leq \frac{5}{2} \bar{d}(Z_A^{\alpha_A}, Z_B^{\alpha_B}) + \frac{5}{2}\eps.
\end{align}

This holds for any $\varepsilon>0$, and thus
\begin{equation}\label{ineq:genMetric2}
\|\alpha_A-\alpha_B\|_\infty \leq \frac52 \bar d(Z_A^{\alpha_A}, Z_B^{\alpha_B}).
\end{equation}

From the continuity and increasing nature of $y_0$, for any $x \in \R$ there exists a $\xi \in \R$ such that $y_0(\xi) = x$. It then follows that
\begin{align*}
	|u_A(x) - u_B(x)| 
	&= |u_A(y_0(\xi)) - u_B(y_0(\xi))|\\
	&\leq |u_A(y_0(\xi)) - u_B(y_N(\xi))| + |u_B(y_N(\xi)) - u_B(y_0(\xi))|\\
	&= |U_0(\xi) - U_N(\xi)| + \left|\int_{y_0(\xi)}^{y_N(\xi)} u_{B,x}(\eta)\, d\eta \right|\\
	&\leq \|U_0 - U_N\|_\infty + \sqrt{|y_0(\xi) - y_N(\xi)|} \left(\int_\R u_{B,x}^2(\eta)\, d\eta\right)^\hlf\\
	&\leq \|X_0^{\alpha_0} - X_N^{\alpha_N}\| +\sqrt{\|X_0^{\alpha_0} - X_N^{\alpha_N}\|}\sqrt{M}\\
	&\leq \frac{5}{2} \bar{d}(Z_A^{\alpha_A}, Z_B^{\alpha_B}) + \frac{5}{2}\eps + \sqrt{M}\sqrt{\frac{5}{2} \bar{d}(Z_A^{\alpha_A}, Z_B^{\alpha_B}) + \frac{5}{2}\eps},
\end{align*}
where we used \eqref{inequ:vik}.
This holds for any $\eps > 0$, and thus
\begin{equation}\label{ineq:genMetric3}
	\|u_A - u_B\|_\infty \leq \frac{5}{2} \bar{d}(Z_A^{\alpha_A}, Z_B^{\alpha_B}) + \sqrt{M}\sqrt{\frac{5}{2} \bar{d}(Z_A^{\alpha_A}, Z_B^{\alpha_B})}.
\end{equation}

Consider any $\phi \in \mathcal{L}$ and $ k = 1,\dots,N $. Then, 
\begin{align*}
	\left| \int_\R \phi(x)\, d(\mu_k - \mu_{k-1})\right|
	&= \left| \int_\R (\phi\circ y_k)(\xi) V_{k,\xi}(\xi) - (\phi \circ y_{k-1})(\xi)V_{k-1,\xi}(\xi)\, d\xi \right|.
\end{align*}
After using $ \xi = f(\eta) $, where $f \in \mG$ is some relabelling function, we find
\begin{align*}
	\left| \int_\R \phi(x)\, d(\mu_k - \mu_{k-1}) \right|
	&= \left| \int_\R (\phi\circ y_k \circ f)(\xi) (V_k \circ f)_\xi(\xi) - (\phi \circ y_{k-1})(\xi)V_{k-1,\xi}(\xi)\, d\xi \right|\\
	&\leq \left| \int_\R (\phi\circ y_k \circ f)(\xi)  ((V_k \circ f)_\xi(\xi) - V_{k-1,\xi}(\xi))\, d\xi \right|\\
	&\quad\ + \left| \int_\R ((\phi\circ y_k \circ f)(\xi) - (\phi \circ y_{k-1})(\xi))V_{k-1,\xi}(\xi)\, d\xi \right|.
\end{align*}
Focusing on the first integral, we have 
\begin{align*}
	\left| \int_\R (\phi\circ y_k \circ f)(\xi)  ((V_k \circ f)_\xi(\xi) - V_{k-1,\xi}(\xi))\, d\xi \right|
	&\leq \|\phi\|_\infty \| (V_k \circ f)_\xi - V_{k-1,\xi} \|_1\\
	&\leq \| (V_k \circ f)_\xi - V_{k-1,\xi} \|_1 \\
	&\leq \|G(X_k^{\alpha_k} \circ f, X_{k-1}^{\alpha_{k-1}})\|_1,
\end{align*}
where the final inequality follows from \eqref{eqn:diffVleqG}
in Proposition \ref{prop:G}.

For the second integral
\begin{align*}
	\left| \int_\R  ((\phi\circ y_k \circ f)(\xi) - (\phi \circ y_{k-1})(\xi))V_{k-1,\xi}(\xi)\, d\xi \right|
	&\leq \int_\R |(y_k \circ f)(\xi) - y_{k-1}(\xi)|V_{k-1,\xi}(\xi)\, d\xi\\
	&\leq M\|y_k \circ f - y_{k-1}\|_\infty
\end{align*}
where we have used $\|\phi\|_{1,\infty} \leq 1$,and $ \|V_{k-1,\xi}\|_1 \leq M $.
Thus after taking the sum of these two inequalities, from a similar argument
to that used for \eqref{eqn:XAcFXB}, we find 
\begin{equation}\label{eqn:phiMu1}
	\left| \int_\R \phi(x)\, d(\mu_k - \mu_{k-1})\right| 
	\leq
	4\bar{M}
	D\left(
		X_k^{\alpha_k} \circ f,
		X_{k-1}^{\alpha_{k-1}}
	\right),
\end{equation}
Swapping the $k$ and $k-1$ terms, and replacing $f$ by another relabelling function $g \in \mG$, we get 
\begin{equation}\label{eqn:phiMu2}
	\left| \int_\R \phi(x)\, d(\mu_k - \mu_{k-1}) \right| 
	\leq 
	4\bar{M}
	D\left(
		X_k^{\alpha_k},
		X_{k-1}^{\alpha_{k-1}} \circ g
	\right).
\end{equation}
Thus, summing~\eqref{eqn:phiMu1} and~\eqref{eqn:phiMu2},
and taking the infimum over all $f, g \in \mG$, we find
\[
	\left| \int_\R \phi(x)\, d(\mu_k - \mu_{k-1}) \right| 
	\leq 2\bar{M} J(X_k^{\alpha_k}, X_{k-1}^{\alpha_{k-1}}),
\]
and hence we can apply the same argument as in the proof of Lemma~\ref{lem:genMetric}
for the left inequality of \eqref{lem:genMetricIneq}, obtaining 
\begin{align*}
	\left|\int_\R \phi(x) d(\mu_k - \mu_{k-1})\right| 
	&\leq 2 \bar{M} \hat{d}(X_k^{\alpha_k}, X_{k-1}^{\alpha_{k-1}})\\
	&= 2\bar{M} d_\mD (Y_k^{\alpha_k}, Y_{k-1}^{\alpha_{k-1}}),
\end{align*}

Taking the infimum over all $\phi \in \mathcal{L}$, and from the definition of $\|\cdot\|_\mathcal{M}$, see \eqref{eqn:BddLip}, we have that
\begin{equation*}
	\|\mu_A - \mu_B\|_\mathcal{M} 
	\leq \sum_{k=1}^N \|\mu_k - \mu_{k-1}\|_\mathcal{M}
	\leq 2 \bar{M} \sum_{k=1}^N d_\mD (Y_k^{\alpha_k}, Y_{k-1}^{\alpha_k-1})
	\leq 2 \bar{M} \bar{d}(Z_A^{\alpha_A}, Z_B^{\alpha_B}) +2\bar{M} \eps.
\end{equation*}
Once again as this construction can be done for any $\eps > 0$, we can conclude
\begin{equation}\label{ineq:genMetric4}
	\|\mu_A - \mu_B\|_\mathcal{M} \leq 2 \bar{M} \bar{d}(Z_A^{\alpha_A}, Z_B^{\alpha_B}).
\end{equation}

Summing up \eqref{ineq:genMetric2}, \eqref{ineq:genMetric3}, and \eqref{ineq:genMetric4}, we get \eqref{ineq:genMetric1}.
\end{proof}

With everything set up, we can finish with our main theorem.
\begin{thm}
Let $Z_A^{\alpha_A}=( (u_A, \mu_A), \alpha_A )$ and $Z_B^{\alpha_B}=( (u_B, \mu_B), \alpha_B )$ be two $ \alpha $-dissipative solutions to
\eqref{eqn:HS}, constructed via the generalised method of characteristics,
with initial data $Z_{0,A}^{\alpha_A}$ and $Z_{0,B}^{\alpha_B}$ in $\mD_{0,M}^L$, respectively.
Then
\begin{equation*}
	\bar{d}(
		Z_A^{\alpha_A}(t), Z_B^{\alpha_B}(t)
	 )\\
	 \leq e^{R_M^{L} t}
	\bar{d}(
		Z_{0,A}^{\alpha_A}, Z_{0,B}^{\alpha_B}
	).
\end{equation*}
with $R_M^{L}$ given by \eqref{const:RM}.
\end{thm}

\begin{proof}
Let $\eps > 0$. Given $Z_{0,A}^{\alpha_A}$ and $Z_{0,B}^{\alpha_B}$ in $\mD_{0,M}^L$,
there exists a sequence   
$$\{Y_{0,k}^{\alpha_k}\}_{k=0}^N=\{(Z_{0,k}, \nu_{0,k}), \alpha_k)\}_{k=0}^N=\{((u_{0,k}, \mu_{0,k}, \nu_{0,k}), \alpha_k)\}_{k=0}^N  \quad \text{ in } \mD_{M}^L$$
such that $Z_{0,0}^{\alpha_0} = Z_{0,A}^{\alpha_A}$, 
$Z_{0,N}^{\alpha_N} = Z_{0,B}^{\alpha_B}$, and
\begin{equation*}
	\sum_{i=k}^N d_\mD(Y_{0,k}^{\alpha_k}, Y_{0,k-1}^{\alpha_{k-1}}) \leq \bar{d}(Z_{0,A}^{\alpha_A}, Z_B^{\alpha_B}) + \eps.
\end{equation*}

Denote by $Y_k^{\alpha_k}$ for $k=0,\dots,N$ the $\alpha$ dissipative solution with initial data $Y_{0,k}^{\alpha_k}$.
Then, from Corollary \ref{cor:lipEstdD},
\begin{align*}
	\bar{d}(Z_A^{\alpha_A}(t), Z_B^{\alpha_B}(t))
	&\leq \sum_{k=1}^N d_\mD(Y_k^{\alpha_k}(t), Y_{k-1}^{\alpha_{k-1}}(t))\\
	&\leq e^{R_M^{L} t} \sum_{k=1}^N d_\mD(Y_{0,k}^{\alpha_k}, Y_{0,k-1}^{\alpha_{k-1}})\\
	&\leq e^{R_M^{L} t} (\bar{d}(Z_{A,0}^{\alpha_A}, Z_B^{\alpha_B}) + \eps ),
\end{align*}
and as this construction can be done for any $\eps >0$, the result holds.
\end{proof}

\subsection{A simplification in the case $\alpha$ is constant.}
Using Section~\ref{subsec:ac} as basis, one can repeat the construction from this section, yielding the following result. For any two $\alpha$-dissipative solutions $Z_A^{\alpha_A}$, $Z_B^{\alpha_B}$ with initial data $Z_{0,A}^{\alpha_A}$, $Z_{0,B}^{\alpha_B}$ in $\mD_0$,
\begin{equation*}
\bar{d}(Z_A^{\alpha_A}(t), Z_B^{\alpha_B}(t))\leq e^{\frac32 t} \bar{d}(Z_{0,A}^{\alpha_A}, Z_{0,B}^{\alpha_B}).
\end{equation*}
Note that here $L=0$ and that the exponent is independent of $M$. This is also why we can consider any initial data in $\mD_0$ and not only in $\mD_{0,M}^{L}$.

\appendix
\section{Important results}
The following result is a well established construction of a pseudo-metric on the quotient of a metric space. For instance, the idea was used in \cite{MR2737213} for the periodic Camassa--Holm equation.
\begin{lem}\label{lem:genMetric}
Let $X \subseteq Y$, with $Y$ a normed space, and suppose
\begin{equation}\label{lem:ineq:Fineq}
	\|x_A - x_B\| \leq C F(x_A, x_B), \quad\  \text{ for all }x_A, x_B \in X,
\end{equation}
for some function $F:X\times X \to \R^+$ and some constant $C >0$. If $F$ satisfies for all $x_A, x_B \in X$
\begin{itemize}
	\item $ x_A = x_B \implies F(x_A, x_B) = 0 $,
	\item $F(x_A, x_B) = F(x_B, x_A)$,
\end{itemize}
then the function $d:X\times X \to \R^+$ given by
\[
	d(x_A, x_B)
	\coloneqq \inf \left\{ \sum_{k=1}^N F(x_k, x_{k-1})\  \bigg\vert\ x_k \in X, x_0 = x_A, x_N = x_B, N \in \N \right\}
\]
is a metric, and 
\begin{equation}\label{lem:genMetricIneq}
	\frac{1}{C}\|x_A - x_B\| \leq d(x_A, x_B) \leq F(x_A, x_B)
\end{equation}
for all $x_A, x_B \in X$.

Should \eqref{lem:ineq:Fineq} not be satisfied, but the rest of the conditions are, then we can only conclude that $d$ is a pseudo-metric. That is, we cannot say $d(x_A, x_B) = 0$ implies $x_A = x_B$, but every other condition of a metric is satisfied.
\end{lem}

\begin{proof} Symmetry is immediate from the assumptions, as well as the fact that if $x_A = x_B$, then $d(x_A, x_B) = 0$.
We begin by showing if $d(x_A, x_B) =0$, then $x_A = x_B$. Let $\eps>0$. Choose a sequence $\{x_k\}_{k=0}^N$ such that
\[
	\sum_{k=1}^N F(x_k, x_{k-1}) \leq d(x_A, x_B)  + \eps.
\]
Then, by our assumption
\[
	\|x_A - x_B\| 
	\leq \sum_{k=1}^N \|x_k - x_{k-1}\| 
	\leq \sum_{k=1}^N C F(x_k, x_{k-1}) \leq C d(x_A, x_B) + C \eps.
\]
This inequality is satisfied for any $\eps>0$, hence
\[
	\|x_A - x_B\| \leq C d(x_A, x_B),
\]
and so, if $d(x_A, x_B) = 0$, $\|x_A - x_B\|=0$. Thus $x_A = x_B$ as required.

The right hand estimate of \eqref{lem:genMetricIneq} is obtained immediately by considering the sequence $x_0 = x_A$ and $x_1 = x_B$ in the definition of $d$.

Next, we have the triangle inequality. Consider $x_A, x_B, x_C \in X$, and let $ \eps > 0 $. Take two sequences, $\{x_k\}_{k=0}^N$ and $ \{x_k\}_{k=N}^M $, with $M>N$, $x_0 = x_A$, $x_N = x_B$ and $x_M = x_C$, such that
\[
	\sum_{k=1}^N F(x_k, x_{k-1}) \leq d(x_A, x_B)  + \eps.
\]     
and
\[
	\sum_{k=N+1}^M F(x_k, x_{k-1}) \leq d(x_B, x_C)  + \eps.
\]
Then
\begin{align*}
	d(x_A, x_C) 
	&\leq \sum_{k=1}^M F(x_k, x_{k-1})\\
	&\leq \sum_{k=1}^N F(x_k, x_{k-1}) + \sum_{k=N+1}^M F(x_k, x_{k-1})\\
	&\leq d(x_A, x_B) + d(x_B, x_C) + 2\eps.
\end{align*}
Hence, as this construction can be done for any $\eps>0$, we have
\[
	d(x_A, x_B) \leq d(x_A, x_B) + d(x_B, x_C),
\]
as required.
\end{proof}
\section{Examples}
We now explore some examples to demonstrate notable details about the constructed metric.

To begin, we note a limitation or advantage of our metric, dependent on ones perspective. Specifically, the role the $\alpha$ plays in the solution is dependent on whether wave breaking actually occurs. Due to the difference of the $\alpha$ measured in our metric, this means one could have a positive distance even if the $u$'s and $\mu$'s are the same for all time.

\begin{exmp}\label{exmp:adiss}
Consider the initial data
\[
	u_0(x) = 
	\begin{cases}
		1 ,&\mbox\ x\leq -2 ,\\
		-1-x ,&\mbox\ -2 < x\leq -1 ,\\
		0 ,&\mbox\ -1 < x\leq 1 ,\\
		1-x ,&\mbox\ 1 < x\leq 2 ,\\
		-1 ,&\mbox\ 2<x,
	\end{cases}
	\quad\
	\mu_0=\nu_0=u_{0,x}^2(x)\, dx,
\]
and from this we can calculate the cumulative energy,
\[
	\mu_0((-\infty,x))=\nu_0((-\infty,x))=
	\begin{cases}
		0, &\mbox\ x\leq -2,\\
		2+x, &\mbox\ -2 < x\leq -1,\\
		1, &\mbox\ -1 < x\leq 1,\\
		x, &\mbox\ 1 < x\leq 2,\\
		2, &\mbox\ 2 < x.
	\end{cases}
\]

Let $\alpha_A\equiv \frac13$, as in Example A.1 in \cite{grunert2021lipschitz}, and $\alpha_B:\R \to [0, 1)$ such that $\alpha_B(1) = \alpha_B(-1) = \frac{1}{3}$, but $\alpha_B \neq \alpha_A$. 

Transforming, using the mapping $ \Lag $ from Definition \eqref{map:EultoLag}, we obtain the initial data in Lagrangian coordinates,
\[
	y_0(\xi) \coloneqq
	\begin{cases}
		\xi ,&\mbox\ \xi\leq -2 ,\\
		-1+\frac{1}{2}\xi ,&\mbox\ -2 < \xi\leq 0 ,\\
		-1+\xi ,&\mbox\ 0 < \xi\leq 2 ,\\
		\frac{1}{2}\xi ,&\mbox\ 2 < \xi\leq 4 ,\\
		-2+\xi ,&\mbox\ 4 < \xi,
	\end{cases}
	\quad\
	U_0(\xi) =
	\begin{cases}
		1 ,&\mbox\ \xi\leq -2 ,\\
		-\frac{1}{2}\xi ,&\mbox\ -2 < \xi\leq 0 ,\\
		0 ,&\mbox\ 0 < \xi\leq 2 ,\\
		1-\frac{1}{2}\xi ,&\mbox\ 2 < \xi\leq 4 ,\\
		-1 ,&\mbox\ 4< \xi,
	\end{cases}
\]
and
\[
	V_0(\xi)=H_0(\xi) =
	\begin{cases}
		0 ,&\mbox\ \xi\leq -2 ,\\
		1+\frac{1}{2}\xi ,&\mbox\ -2 < \xi\leq 0 ,\\
		1 ,&\mbox\ 0 < \xi\leq 2 ,\\
		\frac{1}{2}\xi ,&\mbox\ 2 < \xi\leq 4 ,\\
		2 ,&\mbox\ 4 < \xi.
	\end{cases}
\]
Determining the wave breaking times using \eqref{def:tau}, we get
\[
	\tau(\xi) =
	\begin{cases}
		2 ,&\mbox\ \xi\in (-2,0)\cup (2,4) ,\\
		+\infty ,&\mbox\ \text{otherwise}.
	\end{cases}
\]
We can then calculate the solution using the ODE system \eqref{eqn:LagSys}, and one obtains for either $\alpha_A$ or $\alpha_B$ that
\[
	y(\xi, t) =
		\begin{cases}
			\begin{cases}
		 t - \frac{1}{4}t^2+\xi , &\mbox\ \xi \leq -2,\\
				-1 +\frac18(t-2)^2\xi, &\mbox\ -2 <  \xi \leq 0,\\
    - 1+\xi, &\mbox\ 0 < \xi \leq 2,\\
				t-\frac14 t^2+\frac18(t-2)^2\xi, &\mbox\ 2 < \xi \leq 4,\\
				-2 -t + \frac{1}{4}t^2+\xi, &\mbox\ 4<\xi,\\
			\end{cases} 
			&\mbox\ 0\leq t<2,\\
			\begin{cases}
		 \frac{1}{3}+\frac{2}{3}t-\frac16 t^2+\xi, &\mbox\ \xi \leq -2,\\
				-1 + \frac1{12}(t-2)^2\xi, &\mbox\ -2 < \xi \leq 0,\\
				- 1+\xi &\mbox\ 0 < \xi \leq 2,\\
				\frac13 +\frac23 t -\frac16 t^2+\frac1{12}(t-2)^2\xi, &\mbox\ 2 < \xi \leq 4,\\
				- \frac{7}{3} - \frac{2}{3}t+\frac16 t^2+\xi, &\mbox\ 4< \xi,\\
			\end{cases}
			&\mbox\ 2\leq t,
		\end{cases}
\]
\[
	U(\xi, t) =
	\begin{cases}
		\begin{cases}
			1-\frac{1}{2}t, &\mbox\ \xi \leq -2,\\
			\frac14(t-2)\xi, &\mbox\ -2 < \xi \leq 0,\\
			0, &\mbox\ 0 < \xi \leq 2,\\
			1-\frac12 t+\frac14(t-2)\xi, &\mbox\ 2 < \xi \leq 4,\\
			-1+\frac{1}{2}t, &\mbox\ 4 < \xi,\\
		\end{cases} 
		&\mbox\ 0\leq t<2,\\
		\begin{cases}
			\frac{2}{3}-\frac{1}{3}t, &\mbox\ \xi \leq -2,\\
			\frac16(t-2)\xi, &\mbox\ -2 < \xi \leq 0,\\
			0, &\mbox\ 0 < \xi \leq 2,\\
			\frac23 -\frac13 t+\frac16(t-2)\xi, &\mbox\ 2 < \xi \leq 4,\\
			-\frac{2}{3}+ \frac{1}{3}t, &\mbox\ 4 < \xi,\\
		\end{cases}
		&\mbox\ 2\leq t,
	\end{cases}
\]

\[ H(\xi,t)=H_0(\xi), \quad 0\leq t,\]
and 
\[
	V(\xi, t) =
	\begin{cases}
		H(\xi), &\mbox\ 0\leq t<2,\\
		\begin{cases}
			0, &\mbox\ \xi \leq -2,\\
			\frac{2}{3}+\frac{1}{3}\xi, &\mbox\ -2 < \xi \leq 0,\\
			\frac{2}{3}, &\mbox\ 0 < \xi \leq 2,\\
			\frac{1}{3}\xi, &\mbox\ 2 < \xi \leq 4,\\
			\frac{4}{3}, &\mbox\ 4 < \xi,\\
		\end{cases}
		&\mbox\ 2\leq t.
	\end{cases}
\]

\begin{figure}
\centering
\includegraphics[scale=0.4]{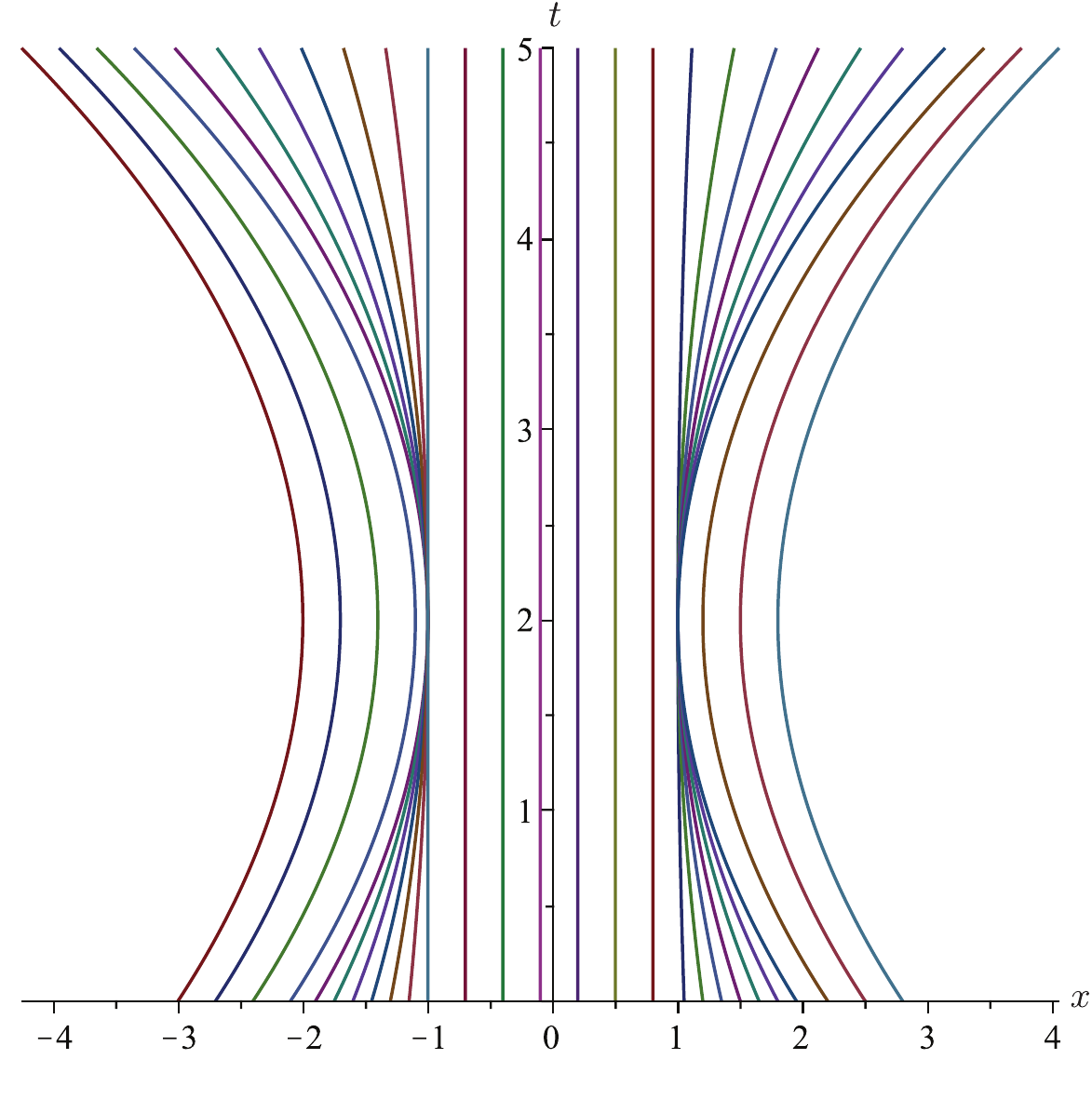}
\caption{
	Plot of the characteristics $y(\xi,t)$ for different values of $\xi$.
	Note the concentration of characteristics at the wave breaking
	time $t=2$, and the subsequent spreading due to only partial
	energy loss.
}
\label{fig:ex2y}
\end{figure}

See Figure \ref{fig:ex2y} for a plot of the characteristics $y$.
\end{exmp}

This example demonstrates that the choice of the metric plays an important role when comparing two solutions. These two solutions remain the same for all time. However the distance given in our metric, constructed using \eqref{eqn:D}, will be positive, as $\alpha_A \neq \alpha_B$. 

This phenomenon occurs if, at points $x \in \R$ where wave breaking occurs, $\alpha_A(x) = \alpha_B(x)$. Or in other words, replacing $\alpha_A$ by $\alpha_B$ or vice versa has no impact on the solutions in that case. 
Therefore, one could argue that following our construction with $D$, given by~\eqref{eqn:D}, replaced by 
\begin{equation*}
	\hat{D}(X_A, X_B)
	= D(X_A^{\alpha_A}, X_B^{\alpha_B}) - \|\alpha_A - \alpha_B\|_\infty,
\end{equation*}
might be more appropriate for certain purposes.

\vspace{0.3cm}

In the next example, we demonstrate why we restrict ourselves from choosing $\alpha:\R \to [0, 1]$, i.e. such that points of wave breaking can be fully dissipative and other points can be partially dissipative or conservative. 
\begin{exmp}\label{exmp:alphFn1}
We consider as initial data,
\begin{equation}
	u_0(x) =
	\begin{cases}
		1, &\mbox\ x\leq 0,\\
		1-x, &\mbox\ 0 < x\leq \hlf,\\
		\frac{3}{2} - 2x, &\mbox\ \hlf < x\leq 1,\\
		-\hlf, &\mbox\ 1<x,
	\end{cases}
	\quad\
	\mu_0 = \nu_0 = u_{0,x}^2\, dx,
\end{equation}
and assume the following values of $\alpha:[0, 1)\to\R$:
\[
	\alpha\left(\frac{13}{16}\right) = 1\quad  \text{ and }\quad \alpha\left(1\right) = \hlf.
\]
The points here are chosen tactically to be where wave breaking occurs in the future.

We begin by calculating the cumulative energy function. We have
\[
	u_{0,x}^2(x) = 
	\begin{cases}
		0, &\mbox\ x\leq 0,\\
		1, &\mbox\ 0 < x\leq \hlf,\\
		4, &\mbox\ \hlf < x\leq 1,\\
		0, &\mbox\ 1<x,
	\end{cases}
\]
and
\[
	\mu_0 \big((-\infty, x)\big) = 
	\nu_0 \big((-\infty, x)\big) = 
	\begin{cases}
		0, &\mbox\ x\leq 0,\\
		x, &\mbox\ 0 < x\leq \hlf,\\
		- \frac{3}{2} + 4x, &\mbox\ \hlf < x\leq 1,\\
		\frac{5}{2}, &\mbox\ 1<x.
	\end{cases}
\]

Thus, using the transformation $\Lag$ from Definition \ref{map:EultoLag},
\begin{equation*}
	y_0(\xi) =
	\begin{cases}
		\xi, &\mbox\ \xi \leq 0,\\
		\frac{1}{2}\xi, &\mbox\ 0 < \xi \leq 1,\\
		\frac{3}{10} + \frac{1}{5}\xi, &\mbox\ 1 < \xi \leq \frac{7}{2},\\
		\xi - \frac{5}{2}, &\mbox\  \frac{7}{2} < \xi,
	\end{cases}
	\quad\ 
	U_0(\xi) =
	\begin{cases}
		1, &\mbox\ \xi \leq 0,\\
		1 - \frac{1}{2}\xi, &\mbox\ 0 < \xi \leq 1,\\
		\frac{9}{10} - \frac{2}{5}\xi, &\mbox\ 1 < \xi \leq \frac{7}{2},\\
		-\hlf, &\mbox\ \frac{7}{2} < \xi,
	\end{cases}
\end{equation*}
and
\begin{equation}
	H_0(\xi) = V_0(\xi) = 
	\begin{cases}
		0, &\mbox\ \xi \leq 0,\\
		\frac{1}{2}\xi, &\mbox\ 0 < \xi \leq 1,\\
		 - \frac{3}{10} + \frac{4}{5}\xi, &\mbox\ 1 < \xi \leq \frac{7}{2},\\
		\frac{5}{2}, &\mbox\ \frac{7}{2} < \xi.
	\end{cases}
\end{equation}
Thus, we can calculate the times at which wave breaking occurs. Using \eqref{def:tau}, 
\[
	\tau(\xi) = \begin{cases}
		2, &\mbox\ \xi\in(0,1),\\
		1, &\mbox\ \xi\in(1, \frac72),\\
		+\infty, &\mbox\ \text{otherwise}.
	\end{cases}
\]

With everything in place, we can solve the ODE system \eqref{eqn:LagSys}, giving
\begin{equation}
	y(\xi, t) =
	\begin{cases}
		\begin{cases}
			 -\frac{5}{16}t^2 + t + \xi, 
				&\mbox\ \xi \leq 0,\\
			 - \frac{5}{16}t^2 + t + \frac{1}{8}(t - 2)^2\xi, 
				&\mbox\ 0 < \xi \leq 1,\\
			\frac{3}{10} + \frac{9}{10}t - \frac{31}{80}t^2 + \frac{1}{5}(t-1)^2\xi
				&\mbox\ 1 < \xi \leq \frac{7}{2},\\
			- \frac{5}{2} - \hlf t + \frac{5}{16}t^2 + \xi,
				&\mbox\ \frac{7}{2} < \xi,
		\end{cases}
		&\mbox\ 0 \leq t < 1,\\
		\begin{cases}
			\frac{1}{4} + \frac{1}{2}t -\frac{1}{16}t^2 + \xi,
				&\mbox\ \xi \leq 0,\\
			\frac{1}{4} + \frac{1}{2} t - \frac{1}{16} t^2 + \frac{1}{8}(t-2)^2\xi ,
				&\mbox\ 0 < \xi \leq 1,\\
			\frac{3}{4} + \frac{1}{16}t^2, 
				&\mbox\ 1 < \xi \leq \frac{7}{2},\\
			- \frac{11}{4} + \frac{1}{16}t^2 + \xi , 
				&\mbox\ \frac{7}{2} < \xi,
		\end{cases}
		&\mbox\ 1 \leq t < 2,\\
		\begin{cases}
			\frac{3}{8} + \frac{3}{8}t - \frac{1}{32}t^2 + \xi,
				&\mbox\ \xi \leq 0,\\
			\frac{3}{8} + \frac{3}{8} t - \frac{1}{32} t^2 + \frac{1}{16}(t-2)^2\xi,
				&\mbox\ 0 < \xi \leq 1,\\
			\frac{5}{8} + \frac{1}{8}t + \frac{1}{32}t^2,
				&\mbox\ 1 < \xi \leq \frac{7}{2},\\
			- \frac{23}{8} + \frac{1}{8}t + \frac{1}{32}t^2 + \xi,
				&\mbox\ \frac{7}{2} < \xi,
		\end{cases}
		&\mbox\ 2 \leq t,
	\end{cases}
\end{equation}
\begin{equation}
	U(\xi, t) =
	\begin{cases}
		\begin{cases}
			1 -\frac{5}{8}t,
				&\mbox\ \xi \leq 0,\\
			1 - \frac{5}{8}t + \frac14(t -2)\xi,
				&\mbox\ 0 < \xi \leq 1,\\
			\frac25(t -1)\xi + \frac{9}{10} - \frac{31}{40}t,
				&\mbox\ 1 < \xi \leq \frac{7}{2},\\
			- \frac{1}{2} + \frac{5}{8}t, 
				&\mbox\  \frac{7}{2} < \xi,
		\end{cases}
		&\mbox\ 0 \leq t < 1,\\
		\begin{cases}
			\frac{1}{2} -\frac{1}{8}t, 
				&\mbox\ \xi \leq 0,\\
			\frac{1}{2} - \frac{1}{8}t + \frac{1}{4}(t - 2)\xi, 
				&\mbox\ 0 < \xi \leq 1,\\
			\frac{1}{8}t, 
				&\mbox\ 1 < \xi,
		\end{cases}
		&\mbox\ 1 \leq t < 2,\\
		\begin{cases}
			\frac{3}{8} -\frac{1}{16}t,
				&\mbox\ \xi \leq 0,\\
			\frac{3}{8} - \frac{1}{16}t + \frac{1}{8}(t - 2)\xi, 
				&\mbox\ 0 < \xi \leq 1,\\
			\frac{1}{8} + \frac{1}{16}t, 
				&\mbox\ 1 < \xi,
		\end{cases}
		&\mbox\ 2 \leq t,
	\end{cases}
\end{equation}
\begin{equation*}
	H(\xi, t) = H_0(\xi), \quad\ 0\leq t, 
\end{equation*}
and
\begin{equation}
	V(\xi, t) =
	\begin{cases}
		H(\xi)
		&\mbox\ 0 \leq t < 1,\\
		\begin{cases}
			0, 
				&\mbox\ \xi \leq 0,\\
			\frac{1}{2}\xi, 
				&\mbox\ 0 < \xi \leq 1,\\
			\hlf, 
				&\mbox\ 1 < \xi,
		\end{cases}
		&\mbox\ 1 \leq t < 2,\\
		\begin{cases}
			0, 
				&\mbox\ \xi \leq 0,\\
			\frac{1}{4}\xi, 
				&\mbox\ 0 < \xi \leq 1,\\
			\frac{1}{4}, 
				&\mbox\ 1 < \xi,
		\end{cases}
		&\mbox\ 2 \leq t.
	\end{cases}
\end{equation}

\begin{figure}
\centering
\includegraphics[scale=0.4]{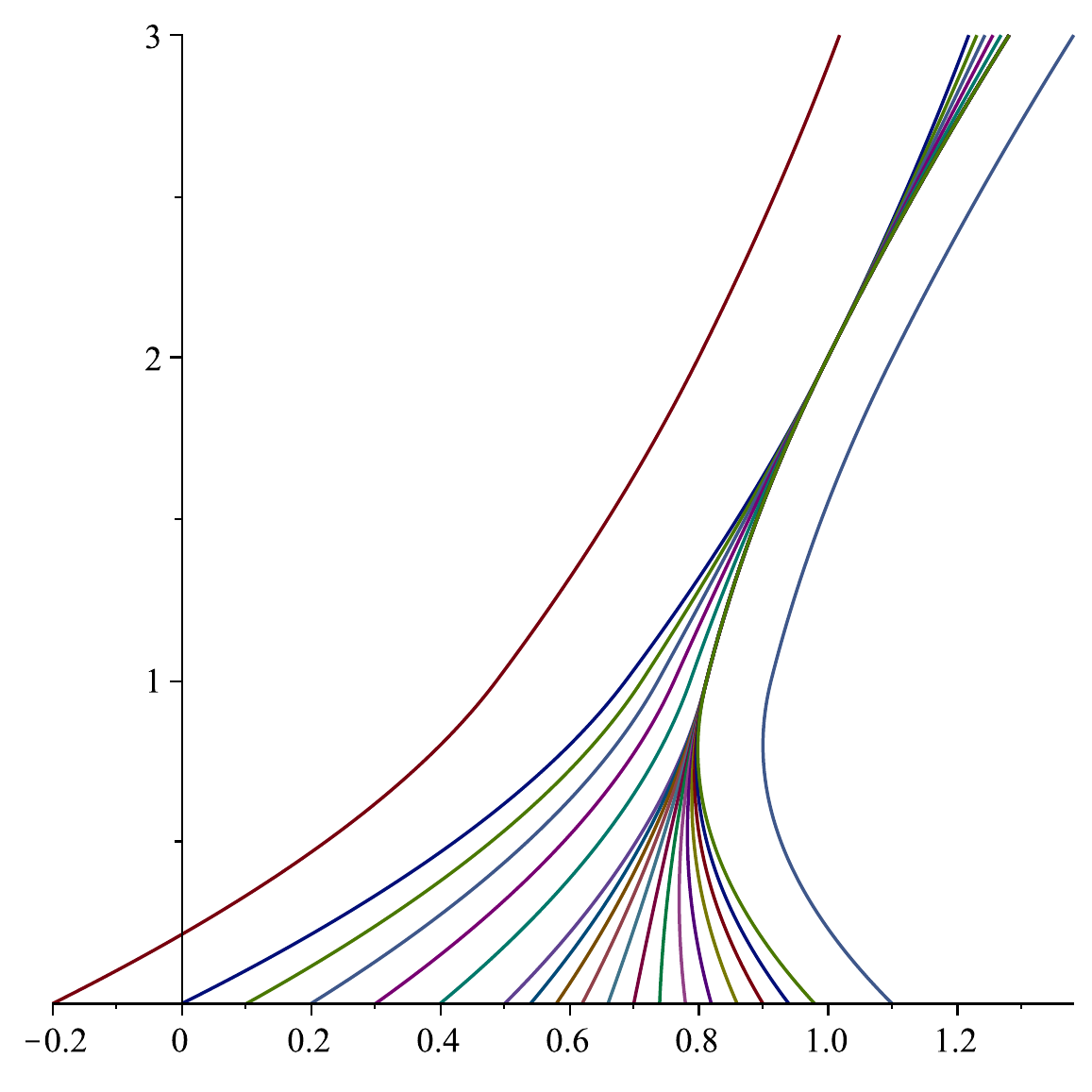}
\caption{
	Plot of the characteristics $y(t,\xi)$
	for different values of $\xi$. 
	In comparison to Figure \ref{fig:ex2y}, there are now two
	wave-breaking times, with the first corresponding to
	full energy dissipation, thus no fan is released,
	and the second to half the energy being lost.
}
\label{fig:ex3y}
\end{figure}

We can transform back into Eulerian coordinates using the mapping $M$ from Definition \ref{map:LagtoEul}, giving at $t = 2$,
\begin{equation}
	\nu \big((-\infty, x), 2\big) = \begin{cases}
		0, &\mbox\ x \leq 1,\\
		\frac{5}{2}, &\mbox\ 1 < x,
	\end{cases}
\end{equation}
\begin{equation}
	\mu \big((-\infty, x), 2\big) = \begin{cases}
		0, &\mbox\ x \leq 1,\\
		\frac{1}{4}, &\mbox\ 1 < x,
	\end{cases}
\end{equation}
and
\begin{equation}
	u(x, 2) = \frac{1}{4}.
\end{equation}

Transforming back to Lagrangian coordinates, setting 
\[
	\bar{X} 
	\coloneqq \Lag\left(u(\cdot, 2), \mu(\cdot, 2), \nu(\cdot ,2)\right),
\] 
we obtain
\begin{equation}
	\bar{y}(\xi) = \begin{cases}
		\xi, &\mbox\ \xi \leq 1,\\
		1, &\mbox\ 1 < \xi \leq \frac{7}{2},\\
		- \frac{5}{2} + \xi, &\mbox\ \frac{7}{2} < \xi,
	\end{cases}
	\quad\
	\bar{U}(\xi) = \frac{1}{4},
\end{equation}
\begin{equation}
	\bar{H}(\xi) = \begin{cases}
		0, &\mbox\ \xi \leq 1,\\
		-1 + \xi, &\mbox\ 1 < \xi \leq \frac{7}{2},\\
		\frac{5}{2}, &\mbox\ \frac{7}{2} < \xi,
	\end{cases}
\end{equation}
and
\begin{equation}
	\bar{V}(\xi) = \begin{cases}
		0, &\mbox\ \xi \leq 1,\\	
		- \frac{1}{10} + \frac{1}{10}\xi, &\mbox\ 1 < \xi \leq \frac{7}{2},\\
		\frac{1}{4}, &\mbox\ \frac{7}{2} < \xi.
	\end{cases}
\end{equation}
And finally we can observe the issue. After transforming to Eulerian coordinates and back, the Lagrangian coordinates are no longer connected by a relabelling function. 

Indeed, one sees that in constructing an $f \in \mathcal{G}$ such that
$\bar{y} \circ f = y(\cdot, 2)$ and $\bar{V}\circ f = V(\cdot, 2)$,
that one must have
\[
	f(\xi) = \begin{cases}
		1 + \xi, &\mbox\ \xi \leq 0,\\
		1 + \frac{5}{2}\xi, &\mbox\ 0 < \xi \leq 1,\\
		\frac{7}{2}, &\mbox\ 1 < \xi \leq \frac{7}{2},\\
		\xi, &\mbox\ \frac{7}{2} < \xi,
	\end{cases}
\]
however, $\bar{H} \circ f \neq H$ and $f\not\in \mathcal{G}$. 
\end{exmp}

In this final example we demonstrate that the choice of $\nu$ has no affect on the final
solution.
\begin{exmp}
Consider as initial data
\begin{equation*}
	u_0(x) = 
	\begin{cases}
		1,	&\mbox\ x \leq -1,\\
		-x, &\mbox\ -1 < x \leq 0, \\
		x, 	&\mbox\ 0 < x \leq 1, \\
		1, 	&\mbox\ 1 < x,
	\end{cases}
\end{equation*}
with
\begin{equation*}
	\mu_0 = u_{0,x}^2\ dx + \delta_{-\hlf} + \delta_{\hlf},
	\quad\ \text{and} \quad\
	\nu_0 = \mu_0 + 3\bbo_{(0,1]} u_{0,x}^2\ dx + \delta_{\hlf}.
\end{equation*}
In this example we consider $\alpha = \hlf$ and drop it from the notation of
coordinates for simplicity.
Set $X_{A,0} = \Lag(u_0, \mu_0, \mu_0)$ and $X_{B,0} = \Lag(u_0, \mu_0, \nu_0)$.
Then
\begin{equation*}
	y_{A,0}(\xi)
	=
	\begin{cases}
		\xi, &\mbox\ \xi \leq -1, \\ 
		- \frac{1}{2} + \frac{1}{2}\xi, &\mbox\ -1 < \xi \leq 0, \\
		- \frac{1}{2}, &\mbox\ 0 < \xi \leq 1, \\
		-1 + \frac{1}{2}\xi, &\mbox\ 1 < \xi \leq 3, \\
		\frac{1}{2}, &\mbox\ 3 < \xi \leq 4, \\
		- \frac{3}{2} + \frac{1}{2}\xi, &\mbox\ 4 < \xi \leq 5, \\
		-4 + \xi, &\mbox\ 5 < \xi,
	\end{cases}
	\quad\
	y_{B,0}(\xi)
	=
	\begin{cases}
		\xi, &\mbox\ \xi \leq -1, \\ 
		- \frac{1}{2} + \frac{1}{2}\xi, &\mbox\ -1 < \xi \leq 0, \\
		- \frac{1}{2}, &\mbox\ 0 < \xi \leq 1, \\
		-1 + \frac{1}{2}\xi, &\mbox\ 1 < \xi \leq 2, \\
		-\frac{2}{5} + \frac{1}{5} \xi, &\mbox\ 2 < \xi \leq \frac{9}{2}, \\
		\frac{1}{2}, &\mbox\ \frac{9}{2} < \xi \leq \frac{13}{2}, \\
		-\frac{4}{5} + \frac{1}{5} \xi, &\mbox\ \frac{13}{2} < \xi \leq 9, \\
		-8 + \xi, &\mbox\ 9 < \xi, 
	\end{cases}
\end{equation*}
\begin{equation*}
	U_{A,0}(\xi)
	=
	\begin{cases}
		1, &\mbox\ \xi \leq -1, \\ 
		\frac{1}{2} - \frac{1}{2}\xi, &\mbox\ -1 < \xi \leq 0, \\
		\frac{1}{2}, &\mbox\ 0 < \xi \leq 1, \\
		1 - \frac{1}{2}\xi, &\mbox\ 1 < \xi \leq 2, \\
		-1 + \frac{1}{2}\xi, &\mbox\ 2 < \xi \leq 3, \\
		\frac{1}{2}, &\mbox\ 3 < \xi \leq 4, \\
		- \frac{3}{2} + \frac{1}{2}\xi, &\mbox\ 4 < \xi \leq 5, \\
		1, &\mbox\ 5 < \xi,
	\end{cases}
	\quad\
	U_{B,0}(\xi)
	=
	\begin{cases}
		1, &\mbox\ \xi \leq -1, \\ 
		\frac{1}{2} - \frac{1}{2}\xi, &\mbox\ -1 < \xi \leq 0, \\
		\frac{1}{2}, &\mbox\ 0 < \xi \leq 1, \\
		1 - \frac{1}{2}\xi, &\mbox\ 1 < \xi \leq 2, \\
		- \frac{2}{5} + \frac{1}{5}\xi, &\mbox\ 2 < \xi \leq \frac{9}{2}, \\
		\frac{1}{2}, &\mbox\ \frac{9}{2} < \xi \leq \frac{13}{2}, \\
		- \frac{4}{5} + \frac{1}{5}\xi, &\mbox\ \frac{13}{2} < \xi \leq 9, \\
		1, &\mbox\ 9 < \xi, 
	\end{cases}
\end{equation*}
\begin{equation*}
	H_{A,0}(\xi)
	=
	\begin{cases}
		0, &\mbox\ \xi \leq -1, \\ 
		\frac{1}{2} + \frac{1}{2}\xi, &\mbox\ -1 < \xi \leq 0, \\
		\frac{1}{2} + \xi, &\mbox\ 0 < \xi \leq 1, \\
		1 + \frac{1}{2}\xi, &\mbox\ 1 < \xi \leq 3, \\
		- \frac{1}{2} + \xi, &\mbox\ 3 < \xi \leq 4, \\
		\frac{3}{2} + \frac{1}{2}\xi, &\mbox\ 4 < \xi \leq 5, \\
		4, &\mbox\ 5 < \xi, \\
	\end{cases}
	\quad\
	H_{B,0}(\xi)
	=
	\begin{cases}
		0, &\mbox\ \xi \leq -1, \\ 
		\frac{1}{2} + \frac{1}{2}\xi, &\mbox\ -1 < \xi \leq 0, \\
		\frac{1}{2} + \xi, &\mbox\ 0 < \xi \leq 1, \\
		1 + \frac{1}{2}\xi, &\mbox\ 1 < \xi \leq 2, \\
		\frac{2}{5} + \frac{4}{5}\xi, &\mbox\ 2 < \xi \leq \frac{9}{2}, \\
		- \frac{1}{2} + \xi, &\mbox\ \frac{9}{2} < \xi \leq \frac{13}{2}, \\
		\frac{4}{5} + \frac{4}{5}\xi, &\mbox\ \frac{13}{2} < \xi \leq 9, \\
		8, &\mbox\ 9 < \xi, 
	\end{cases}
\end{equation*}
and
\begin{equation*}
	V_{A,0}(\xi) = H_{A,0}(\xi),
	\quad\
	V_{B,0}(\xi)
	=
	\begin{cases}
		0, &\mbox\ \xi \leq -1, \\ 
		\frac{1}{2} + \frac{1}{2}\xi, &\mbox\ -1 < \xi \leq 0, \\
		\frac{1}{2} + \xi, &\mbox\ 0 < \xi \leq 1, \\
		1 + \frac{1}{2}\xi, &\mbox\ 1 < \xi \leq 2, \\
		\frac{8}{5} + \frac{1}{5}\xi, &\mbox\ 2 < \xi \leq \frac{9}{2}, \\
		\frac{1}{4} + \frac{1}{2} \xi, &\mbox\ \frac{9}{2} < \xi \leq \frac{13}{2}, \\
		\frac{11}{5} + \frac{1}{5}\xi, &\mbox\ \frac{13}{2} < \xi \leq 9, \\
		4, &\mbox\ 9 < \xi.
	\end{cases}
\end{equation*}
Calculating $\tau(\xi)$ via~\eqref{def:tau} for both sets of initial data, one finds
\begin{equation*}
\tau_A(\xi)=\tau_B(\xi)=\begin{cases} 2, & \xi  \in (-1, 0) \cup (1, 2),\\
+\infty, &\text{otherwise}.
\end{cases}
\end{equation*}

Solving the ODE system~\eqref{eqn:LagSys} with this initial data, one finds
\begin{equation*}
	V_A(\xi, t) =
	\begin{cases}
		V_{A,0}(\xi), &\mbox\ t < 2 \\
		\begin{cases}
			0, &\mbox\ \xi \leq -1, \\ 
			\frac{1}{4} + \frac{1}{4}\xi, &\mbox\ -1 < \xi \leq 0, \\
			\frac{1}{4} + \xi, &\mbox\ 0 < \xi \leq 1, \\
			1 + \frac{1}{4}\xi, &\mbox\ 1 < \xi \leq 2, \\
			\frac{1}{2} + \frac{1}{2}\xi, &\mbox\ 2 < \xi \leq 3, \\
			-1 + \xi, &\mbox\ 3 < \xi \leq 4,\\
			1 + \frac{1}{2} \xi, &\mbox\ 4 < \xi \leq 5,\\
			\frac{7}{2}, &\mbox\ 5 < \xi,
		\end{cases}
									&\mbox\ 2 \leq t.
	\end{cases}
\end{equation*}
\begin{equation}
	U_A(\xi,t) =
	\begin{cases}
		\begin{cases}
			1 - t,
				& \xi \leq -1, \\ 
			\frac{1}{2} - \frac{3}{4} t + \frac{1}{4} (t-2) \xi,
				& -1 < \xi \leq 0, \\
			\frac{1}{2} - \frac{3}{4} t + \frac{1}{2} t \xi,
				& 0 < \xi \leq 1, \\
			1 - \frac{1}{2} t + \frac{1}{4} (t-2) \xi,
				& 1 < \xi \leq 2, \\
			-1 - \frac{1}{2} t + \frac{1}{4} (t+2) \xi,
				& 2 < \xi \leq 3, \\
			\frac{1}{2} - \frac{5}{4} t + \frac{1}{2} t \xi,
				& 3 < \xi \leq 4, \\
			-\frac{3}{2} - \frac{1}{4} t + \frac{1}{4} (t+2) \xi,
				& 4 < \xi \leq 5, \\
			1 + t, & 5 < \xi,
		\end{cases}
			&\mbox\ t < 2,\\
		\begin{cases}
			\frac{3}{4} - \frac{7}{8} t,
				& \xi \leq -1, \\ 
			\frac{1}{2} - \frac{3}{4} t + \frac{1}{8} (t-2) \xi,
				& -1 < \xi \leq 0, \\
			\frac{1}{2} - \frac{3}{4} t + \frac{1}{2} t \xi,
				& 0 < \xi \leq 1, \\
			\frac{3}{4} - \frac{3}{8} t + \frac{1}{8} (t-2) \xi,
				& 1 < \xi \leq 2, \\
			-\frac{3}{4} - \frac{5}{8} t + \frac{1}{4}(t+2) \xi,
				& 2 < \xi \leq 3, \\
			\frac{3}{4} - \frac{11}{8} t + \frac{1}{2} t \xi,
				& 3 < \xi \leq 4, \\
			-\frac{5}{4} - \frac{3}{8} t + \frac{1}{4} (t+2) \xi,
				& 4 < \xi \leq 5, \\
			\frac{5}{4} + \frac{7}{8} t,
				& 5 < \xi,
		\end{cases}
			&\mbox\ 2 \leq t,\\
	\end{cases}
\end{equation}
\begin{equation}
	y_A(\xi,t) =
	\begin{cases}
		\begin{cases}
			t - \frac{1}{2}t^2 + \xi,
				& \xi \leq -1, \\ 
			- \frac{1}{2} + \frac{1}{2}t - \frac{3}{8}t^2 + \frac{1}{8} (t-2)^2 \xi,
				& -1 < \xi \leq 0, \\
			- \frac{1}{2}  + \frac{1}{2} t - \frac{3}{8} t^2 + \frac{1}{4} t^2 \xi,
				& 0 < \xi \leq 1, \\
			-1 + t - \frac{1}{4} t^2 + \frac{1}{8} (t-2)^2 \xi,
				& 1 < \xi \leq 2, \\
			-1 - t - \frac{1}{4} t^2 + \frac{1}{8} (t+2)^2 \xi,
				& 2 < \xi \leq 3, \\
			\frac{1}{2} + \frac{1}{2} t - \frac{5}{8} t^2 + \frac{1}{4} t^2 \xi,
				& 3 < \xi \leq 4, \\
			-\frac{3}{2} - \frac{3}{2} t - \frac{1}{8} t^2 + \frac{1}{8} (t+2)^2 \xi,
				& 4 < \xi \leq 5, \\
			-4 + t + \frac{1}{2} t^2 + \xi,
				& 5 < \xi,
		\end{cases}
			&\mbox\ t < 2,\\
		\begin{cases}
			\frac{1}{4} + \frac{3}{4} t - \frac{7}{16} t^2 + \xi,
				& \xi \leq -1, \\ 
			- \frac{1}{2} + \frac{1}{2} t - \frac{3}{8} t^2 + \frac{1}{16} (t-2)^2 \xi,
				& -1 < \xi \leq 0, \\
			- \frac{1}{2} + \frac{1}{2} t - \frac{3}{8} t^2 + \frac{1}{4} t^2 \xi,
				& 0 < \xi \leq 1, \\
			- \frac{3}{4} + \frac{3}{4} t - \frac{3}{16} t^2 + \frac{1}{16} (t-2)^2 \xi,
				& 1 < \xi \leq 2, \\
			- \frac{5}{4} - \frac{3}{4} t - \frac{5}{16} t^2 + \frac{1}{8} (t+2)^2 \xi,
				& 2 < \xi \leq 3, \\
			\frac{1}{4} + \frac{3}{4} t - \frac{11}{16} t^2 + \frac{1}{4} t^2 \xi,
				& 3 < \xi \leq 4, \\
			- \frac{7}{4} - \frac{5}{4} t - \frac{3}{16} t^2 + \frac{1}{8} (t + 2)^2 \xi,
				& 4 < \xi \leq 5, \\
			-\frac{17}{4} + \frac{5}{4} t + \frac{7}{16} t^2 + \xi,
				& 5 < \xi,
		\end{cases}
			&\mbox\ 2 \leq t.\\
	\end{cases}
\end{equation}
One then finds
\begin{equation}\label{exmp:u}
	u(x,t) = 
	\begin{cases}
		\begin{cases}
			1 - t,
				& x \leq -1 + t - \frac{1}{2} t^2,\\
			\frac{t + 2x}{(t-2)},
				& -1 + t - \frac{1}{2} t^2 < x \leq - \frac{1}{2} + \frac{1}{2}t - \frac{3}{8}t^2,\\
			\frac{2 - t + 4x}{2t},
				& - \frac{1}{2} + \frac{1}{2}t - \frac{3}{8}t^2 < x \leq - \frac{1}{2} + \frac{1}{2}t - \frac{1}{8} t^2,\\
			\frac{2x}{t-2},
				& - \frac{1}{2} + \frac{1}{2}t - \frac{1}{8} t^2 < x \leq 0,\\
			\frac{2x}{t + 2},
				& 0 < x \leq \frac{1}{2} + \frac{1}{2}t + \frac{1}{8}t^2, \\
			\frac{-2 - t + 4x}{2t},
				& \frac{1}{2} + \frac{1}{2}t + \frac{1}{8}t^2 < x \leq \frac{1}{2} + \frac{1}{2}t + \frac{3}{8}t^2, \\
			\frac{t + 2x}{t+2},
				& \frac{1}{2} + \frac{1}{2}t + \frac{3}{8}t^2 < x \leq 1 + t + \frac{1}{2}t^2,\\
			1 + t,
				& 1 + t + \frac{1}{2}t^2 < x,
		\end{cases}
			&\mbox\ t < 2, \\
		\begin{cases}
			\frac{3}{4} - \frac{7}{8}t,
				& x \leq - \frac{3}{4} + \frac{3}{4}t - \frac{7}{16}t^2,\\
			\frac{t + 2x}{t - 2},
				& - \frac{3}{4} + \frac{3}{4}t - \frac{7}{16}t^2 < x \leq - \frac{1}{2} + \frac{1}{2}t - \frac{3}{8}t^2, \\
			\frac{2 - t + 4x}{2t},
				& - \frac{1}{2} + \frac{1}{2}t - \frac{3}{8}t^2 < x \leq - \frac{1}{2} + \frac{1}{2} t - \frac{1}{8}t^2, \\
			\frac{2x}{(t-2)},
				& - \frac{1}{2} + \frac{1}{2} t - \frac{1}{8}t^2 < x \leq - \frac{1}{4} + \frac{1}{4}t - \frac{1}{16}t^2, \\
			\frac{2 - t + 4x}{2(t+2)},
				& - \frac{1}{4} + \frac{1}{4}t - \frac{1}{16}t^2 < x \leq \frac{1}{4} + \frac{3}{4}t + \frac{1}{16}t^2, \\
			\frac{-2 - 3t + 8x}{4t},
				& \frac{1}{4} + \frac{3}{4}t + \frac{1}{16}t^2 < x \leq \frac{1}{4} + \frac{3}{4}t + \frac{5}{16}t^2, \\
			\frac{2 + t + 4x}{2(t+2)},
				& \frac{1}{4} + \frac{3}{4}t + \frac{5}{16}t^2 < x \leq \frac{3}{4} + \frac{5}{4}t + \frac{7}{16}t^2, \\
			\frac{5}{4} + \frac{7}{8}t,
				& \frac{3}{4} + \frac{5}{4}t + \frac{7}{16}t^2<x,
		\end{cases}
			&\mbox\ 2 \leq t,
	\end{cases}
\end{equation}
and, for $t >0$,
\begin{equation}\label{exmp:mu}
	\mu(t) = u_x^2(\cdot, t)\ dx 
	+ \frac{1}{4}(\delta_{-1} + \delta_{0})\bbo_{\{t=2\}}(t).
\end{equation}

On the other hand,
\begin{equation*}
	V_B(\xi ,t) 
	=
	\begin{cases}
		V_{B,0}(\xi), 
			&\mbox\ t < 2, \\ 
		\begin{cases}
			0,
				& \xi \leq -1, \\ 
			\frac{1}{4} + \frac{1}{4} \xi,
				& -1 < \xi \leq 0, \\
			\frac{1}{4} + \xi,
				& 0 < \xi \leq 1, \\
			1 + \frac{1}{4} \xi,
				& 1 < \xi \leq 2, \\
			\frac{11}{10} + \frac{1}{5} \xi,
				& 2 < \xi \leq \frac{9}{2}, \\
			-\frac{1}{4} + \frac{1}{2} \xi,
				& \frac{9}{2} < \xi \leq \frac{13}{2}, \\
			\frac{17}{10} + \frac{1}{5} \xi,
				& \frac{13}{2} < \xi \leq 9, \\
			\frac{7}{2},
				& 9 < \xi,
		\end{cases}
			&\mbox\ 2 \leq t, \\
	\end{cases}
\end{equation*}
\begin{equation*}
	U_B(\xi,t)
	=
	\begin{cases}
		\begin{cases}
			1 - t,
				& \xi \leq -1, \\ 
			\frac{1}{2} - \frac{3}{4} t + \frac{1}{4} (t-2) \xi,
				& -1 < \xi \leq 0, \\
			\frac{1}{2} - \frac{3}{4} t + \frac{1}{2} t \xi,
				& 0 < \xi \leq 1, \\
			1 - \frac{1}{2} t + \frac{1}{4} (t-2) \xi,
				& 1 < \xi \leq 2, \\
			-\frac{2}{5} - \frac{1}{5} t + \frac{1}{10} (t + 2) \xi,
				& 2 < \xi \leq \frac{9}{2}, \\
			\frac{1}{2} - \frac{7}{8} t + \frac{1}{4} t \xi,
				& \frac{9}{2} < \xi \leq \frac{13}{2}, \\
			- \frac{4}{5} + \frac{1}{10} t + \frac{1}{10} (t+2) \xi,
				& \frac{13}{2} < \xi \leq 9, \\
			1 + t,
				& 9 < \xi,
		\end{cases}
			&\mbox\ t < 2, \\
		\begin{cases}
				\frac{3}{4} - \frac{7}{8} t,
				& \xi \leq -1, \\ 
				\frac{1}{2} - \frac{3}{4} t + \frac{1}{8} (t-2) \xi,
				& -1 < \xi \leq 0, \\
				\frac{1}{2} - \frac{3}{4} t + \frac{1}{2} t \xi,
				& 0 < \xi \leq 1, \\
			\frac{3}{4} - \frac{3}{8} t + \frac{1}{8} (t-2) \xi,
				& 1 < \xi \leq 2, \\
			- \frac{3}{20} - \frac{13}{40} t + \frac{1}{10} (t+2) \xi,
				& 2 < \xi \leq \frac{9}{2}, \\
				\frac{3}{4} - t + \frac{1}{4} t \xi,
				& \frac{9}{2} < \xi \leq \frac{13}{2}, \\
			- \frac{11}{20} - \frac{1}{40} t + \frac{1}{10} (t+2) \xi,
				& \frac{13}{2} < \xi \leq 9, \\
				\frac{5}{4} + \frac{7}{8} t,
				& 9 < \xi,
		\end{cases}
			&\mbox\ 2 \leq t, \\
	\end{cases}
\end{equation*}
\begin{equation*}
	y_B(\xi,t)
	=
	\begin{cases}
		\begin{cases}
			t - \frac{1}{2}t^2 + \xi,
				& \xi \leq -1, \\ 
			-\frac{1}{2} + \frac{1}{2} t - \frac{3}{8} t^2 + \frac{1}{8} (t-2)^2 \xi,
				& -1 < \xi \leq 0, \\
			- \frac{1}{2} + \frac{1}{2} t - \frac{3}{8} t^2 + \frac{1}{4} t^2 \xi,
				& 0 < \xi \leq 1, \\
			-1 + t - \frac{1}{4} t^2 + \frac{1}{8} (t-2)^2 \xi,
				& 1 < \xi \leq 2, \\
			- \frac{2}{5} - \frac{2}{5} t - \frac{1}{10} t^2 + \frac{1}{20} (t+2)^2 \xi,
				& 2 < \xi \leq \frac{9}{2}, \\
			\frac{1}{2} + \frac{1}{2} t - \frac{7}{16} t^2 + \frac{1}{8} t^2 \xi,
				& \frac{9}{2} < \xi \leq \frac{13}{2}, \\
			-\frac{4}{5} - \frac{4}{5} t + \frac{1}{20} t^2 + \frac{1}{20} (t+2)^2 \xi,
				& \frac{13}{2} < \xi \leq 9, \\
			-8 + t + \frac{1}{2} t^2 + \xi,
				& 9 < \xi,
		\end{cases}
			&\mbox\ t < 2, \\
		\begin{cases}
			\frac{1}{4} + \frac{3}{4} t - \frac{7}{16} t^2 + \xi,
				& \xi \leq -1, \\ 
			-\frac{1}{2} + \frac{1}{2} t - \frac{3}{8} t^2 + \frac{1}{16} (t-2)^2 \xi,
				& -1 < \xi \leq 0, \\
			-\frac{1}{2} + \frac{1}{2} t - \frac{3}{8} t^2 + \frac{1}{4} t^2 \xi,
				& 0 < \xi \leq 1, \\
			-\frac{3}{4} + \frac{3}{4} t - \frac{3}{16} t^2 + \frac{1}{16} (t-2)^2 \xi,
				& 1 < \xi \leq 2, \\
			-\frac{13}{20} - \frac{3}{20} t - \frac{13}{80} t^2 + \frac{1}{20} (t+2)^2 \xi,
				& 2 < \xi \leq \frac{9}{2}, \\
			\frac{1}{4} + \frac{3}{4} t - \frac{1}{2} t^2 + \frac{1}{8} t^2 \xi,
				& \frac{9}{2} < \xi \leq \frac{13}{2}, \\
			-\frac{21}{20} - \frac{11}{20} t - \frac{1}{80} t^2 + \frac{1}{20} (t+2)^2 \xi,
				& \frac{13}{2} < \xi \leq 9, \\
			-\frac{33}{4} + \frac{5}{4} t + \frac{7}{16} t^2 + \xi,
				& 9 < \xi,
		\end{cases}
			&\mbox\ 2 \leq t, \\
	\end{cases}
\end{equation*}
and one sees that the transformation $M$ yields again $(u,\mu)$ given by ~\eqref{exmp:u} and \eqref{exmp:mu}.
\end{exmp}

\end{document}